\documentclass[12pt]{article}

\usepackage[T1]{fontenc}
\usepackage{amsfonts}
\usepackage{amscd}
\usepackage{amsmath}
\usepackage{epsfig}
\usepackage{amsthm}
\usepackage{amssymb}
\usepackage{amsbsy}
\usepackage{latexsym}
\usepackage{eucal}
\usepackage{mathrsfs,bm}
\usepackage{mathtools}
\usepackage[all, knot]{xy}
\usepackage{tikz}
\usetikzlibrary{arrows.meta}
\usepackage[colorlinks]{hyperref}
\usepackage{color}

\xyoption{arc}
\xyoption{color}
\xyoption{line}

\newtheorem{theorem}{Theorem}[section]

\newtheorem{proposition}[theorem]{Proposition}
\newtheorem{corollary}[theorem]{Corollary}
\newtheorem{remark}[theorem]{Remark}
\newtheorem{lemma}[theorem]{Lemma}

\def\be{\begin{equation}}
\def\ee{\end{equation}}
\def\ba{\begin{aligned}}
\def\ea{\end{aligned}}
\def\ben{\begin{displaymath}}
\def\een{\end{displaymath}}
\def\baa{\begin{eqnarray}}
\def\eaa{\end{eqnarray}}

\renewcommand{\leq}{\leqslant}
\renewcommand{\geq}{\geqslant}

\newcommand{\Ind}{\operatorname{Ind}}

\newcommand{\Area}{\operatorname{Area}}
\newcommand{\Reg}{\operatorname{Reg}}

\makeatletter
\@addtoreset{equation}{section}
\makeatother

\title{\bfseries Spectral determinants of the Bolza surface and the Klein quartic}
\author{Victor Kalvin}
\date{}

\begin{document}

\maketitle

\begin{abstract}
We obtain closed explicit formulas for the spectral determinants of the smooth hyperbolic Bolza
surface and the Klein quartic. In each case, a multiplicative relation expresses the determinant of
the surface in terms of determinants of singular quotient orbifolds of genera zero and one. The
elliptic factors are evaluated by applying the singular Polyakov anomaly formula to explicit Belyi
maps on CM elliptic curves of discriminants $-8$ and $-7$, while the genus-zero factors are
evaluated by explicit determinant formulas for constant-curvature spheres with conical singularities. The same
multiplicative relations hold fibrewise on the corresponding equisymmetric deformation strata and
yield determinant and first-variation identities. We also prove that every compact quasiplatonic
hyperbolic surface is a critical point of the spectral determinant on its Teichm\"uller space; in
particular, this applies to the Bolza surface and the Klein quartic.
\end{abstract}

\section{Introduction}

For a closed Riemannian surface $M$, let
\be\label{ZetaDet}
 \det\Delta_M=\exp\bigl(-\zeta_M'(0)\bigr),
 \qquad
 \zeta_M(s)=\sum_{\lambda>0}\lambda^{-s},
\ee
where the sum is over the positive eigenvalues of the Friedrichs Laplacian $\Delta_M$. We find
closed explicit formulas for $\det\Delta_M$ when $M$ is the Bolza surface or the Klein quartic
with its smooth hyperbolic metric. These are the maximally symmetric compact Riemann surfaces of
genera two and three, the two lowest genera admitting smooth hyperbolic metrics. To the best of our
knowledge, no spectral determinant of a smooth compact hyperbolic surface has previously been
evaluated in closed form. The classical determinant--Selberg-zeta identity
\cite{DHokerPhong,SarnakDeterminants} expresses $\det\Delta_M$, up to a known factor depending
only on the genus, as $Z'_{\mathrm{Sel},M}(1)$, a global invariant determined by the primitive
length spectrum, but does not evaluate the latter.  Consequently, our formulas
also give closed explicit values of $Z'_{\mathrm{Sel},B}(1)$ and $Z'_{\mathrm{Sel},K}(1)$.

This work continues a program initiated in \cite{KalvinJFA} with the Polyakov anomaly formula for
metrics with cone singularities and developed further in \cite{KalvinCV,KalvinASNS} through explicit
uniformization by Belyi functions. For the corresponding formula for smooth metrics, see, e.g.,
\cite{OsgoodPhillipsSarnak}. The singular anomaly formula expresses the determinant in terms of a
reference determinant, a global conformal integral, and explicit local contributions at the cone
points. It leads to a closed evaluation whenever these data are known explicitly.

The basic solvable geometry is the constant-curvature sphere with three cone singularities. In
\cite{KalvinCV}, its determinant and the regular parts of its conformal factor were found in closed
form. The crucial step was the explicit evaluation of the classical Liouville action, which supplies
the global integral in the anomaly formula. This was achieved in \cite[Section~3]{KalvinCV} with the
help of the Zamolodchikov--Zamolodchikov variational relations \cite{ZamolodchikovZamolodchikov}.

In \cite{KalvinASNS}, the program was extended from the three-cone sphere to Platonic surfaces by
using Belyi functions, that is, holomorphic maps to $\mathbb P^1$ whose critical values lie in
$\{0,1,\infty\}$. In the Platonic cases the required Belyi functions are the classical polyhedral
functions found by Schwarz and Klein \cite{Schwarz,KleinIcosahedron}. The corresponding pullback of
a constant-curvature three-cone metric produces the prescribed constant-curvature metrics with cone
singularities on the Platonic surfaces. The ramification data determine the resulting cone angles,
and the singular anomaly formula transports the determinant together with the local and global
conformal data through the covering. This yields explicit determinant formulas for spherical,
Euclidean, and hyperbolic Platonic surfaces and their singular deformations~\cite{KalvinASNS}. For
the resulting genus-zero pullbacks, the accessory parameters are explicit
\cite[Lemma~4.1]{KalvinASNS}; for the classical relation between accessory parameters and the
Liouville action, see \cite{TakhtajanZograf}.

In the present paper we pass from these singular model geometries to smooth compact hyperbolic
surfaces. Direct uniformization in higher genus does not provide the global and local conformal data
required by the anomaly formula in closed form. We avoid this problem by reducing the determinant of
the original surface to determinants of singular quotient orbifolds. Namely, a virtual relation
between permutation representations of the isometry group gives, by Frobenius reciprocity, an exact
relation between the corresponding spectral zeta functions, and hence a multiplicative relation
between their determinants. This is closely related to Sunada's argument and may be viewed as a
spectral version of Artin formalism; see \cite{Sunada,VenkovZograf}.

For the Bolza surface, a Klein four-group reduces the determinant to one elliptic quotient $T_B$
with two cone points and two genus-zero quotients. For the Klein quartic, a permutation-character
identity in $PSL(2,7)$ leaves an elliptic quotient $T_K$ and two three-cone spheres. In both
cases the resulting quotient determinants can be evaluated explicitly. The genus-zero factors lie in
the three-cone and Platonic framework of \cite{KalvinCV,KalvinASNS}, while the elliptic factors are
singular CM tori of discriminants $-8$ and $-7$. The quotient lattices and the resulting
determinant identities are summarized in Figure~\ref{fig:quotient-lattices}.

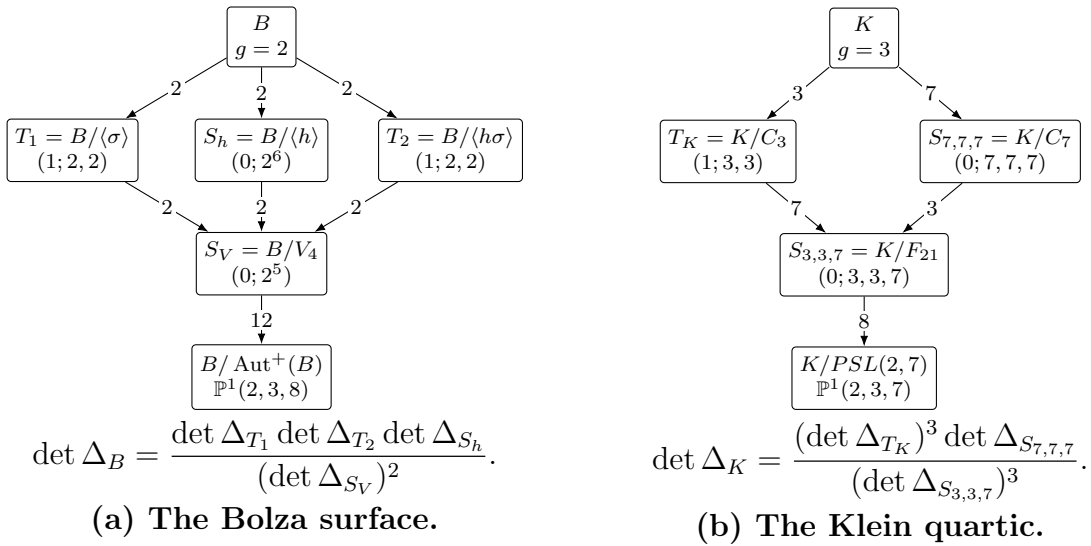
\begin{figure}[t]
\centering
\begin{minipage}[t]{0.49\textwidth}
\centering
\begin{tikzpicture}[
  x=1cm,y=1cm,
  quotient/.style={
    draw,
    rounded corners=1.5pt,
    align=center,
    inner sep=3pt,
    font=\scriptsize
  },
  map/.style={-{Latex[length=4pt]},thin},
  degree/.style={font=\scriptsize,fill=white,inner sep=1pt}
]
\node[quotient] (B) at (0,4.2)
  {$B$\\ $g=2$};
\node[quotient] (T1) at (-2.5,2.7)
  {$T_1=B/\langle\sigma\rangle$\\ $(1;2,2)$};
\node[quotient] (Sh) at (0,2.7)
  {$S_h=B/\langle h\rangle$\\ $(0;2^6)$};
\node[quotient] (T2) at (2.5,2.7)
  {$T_2=B/\langle h\sigma\rangle$\\ $(1;2,2)$};
\node[quotient] (SV) at (0,1.2)
  {$S_V=B/V_4$\\ $(0;2^5)$};
\node[quotient] (Q238) at (0,-0.3)
  {$B/\operatorname{Aut}^+(B)$\\
   $\mathbb P^1(2,3,8)$};

\draw[map] (B) -- node[degree] {$2$} (T1);
\draw[map] (B) -- node[degree] {$2$} (Sh);
\draw[map] (B) -- node[degree] {$2$} (T2);
\draw[map] (T1) -- node[degree] {$2$} (SV);
\draw[map] (Sh) -- node[degree] {$2$} (SV);
\draw[map] (T2) -- node[degree] {$2$} (SV);
\draw[map] (SV) -- node[degree] {$12$} (Q238);
\end{tikzpicture}

\smallskip
$\displaystyle
 \det\Delta_B=
 \frac{\det\Delta_{T_1}\det\Delta_{T_2}\det\Delta_{S_h}}
      {(\det\Delta_{S_V})^2}.
$

\smallskip
\textbf{(a) The Bolza surface.}
\end{minipage}\hfill
\begin{minipage}[t]{0.49\textwidth}
\centering
\begin{tikzpicture}[
  x=1cm,y=1cm,
  quotient/.style={
    draw,
    rounded corners=1.5pt,
    align=center,
    inner sep=3pt,
    font=\scriptsize
  },
  map/.style={-{Latex[length=4pt]},thin},
  degree/.style={font=\scriptsize,fill=white,inner sep=1pt}
]
\node[quotient] (K) at (0,4.2)
  {$K$\\ $g=3$};
\node[quotient] (TK) at (-1.8,2.7)
  {$T_K=K/C_3$\\ $(1;3,3)$};
\node[quotient] (S777) at (1.8,2.7)
  {$S_{7,7,7}=K/C_7$\\ $(0;7,7,7)$};
\node[quotient] (S337) at (0,1.2)
  {$S_{3,3,7}=K/F_{21}$\\ $(0;3,3,7)$};
\node[quotient] (Q237) at (0,-0.3)
  {$K/PSL(2,7)$\\
   $\mathbb P^1(2,3,7)$};

\draw[map] (K) -- node[degree] {$3$} (TK);
\draw[map] (K) -- node[degree] {$7$} (S777);
\draw[map] (TK) -- node[degree] {$7$} (S337);
\draw[map] (S777) -- node[degree] {$3$} (S337);
\draw[map] (S337) -- node[degree] {$8$} (Q237);
\end{tikzpicture}

\smallskip
$\displaystyle
 \det\Delta_K=
 \frac{(\det\Delta_{T_K})^3\det\Delta_{S_{7,7,7}}}
      {(\det\Delta_{S_{3,3,7}})^3}.
$

\smallskip
\textbf{(b) The Klein quartic.}
\end{minipage}
\caption{The quotient lattices used in the two spectral reductions.
An arrow labelled $d$ is a quotient map of degree $d$.
The arrows describe the geometric factorizations through symmetry
quotients.  The exponents in the determinant identities do not arise
from the degrees of these maps: they are determined by the virtual
permutation-representation relations
\eqref{eq:v4-relation} and
\eqref{eq:klein-character-relation}, applied separately on each
eigenspace of the relevant Laplacian.}
\label{fig:quotient-lattices}
\end{figure}

To evaluate the genus-one factors, we use the singular Polyakov anomaly formula \cite{KalvinJFA} to
compare the determinant of the singular hyperbolic metric with that of the flat metric of a
holomorphic differential. The Liouville equation and the norm of the derivative of the Belyi map
reduce the global integral and the local coefficients on the torus to the three-cone data found
explicitly in \cite{KalvinCV}. The only remaining spectral input is the determinant of the flat
reference torus $\mathbb C/(\mathbb Z+\tau\mathbb Z)$, where $\tau\in\mathbb H$ is the ratio of
its periods after one period has been normalized to $1$. Algebraic models from
\cite{FiteSutherland,KoziarzRitoRoulleau} and \cite{HoshinoNakamura} give the period ratios
$$
 \tau_B=i\sqrt2
 \quad\text{and}\quad
 \tau_K=\frac{1+i\sqrt7}{2}
$$
for $T_B$ and $T_K$, respectively. The Kronecker limit formula expresses the determinants of the
flat reference tori explicitly in terms of their areas, period ratios, and the Dedekind eta
function; see, e.g., \cite{OsgoodPhillipsSarnak}.

Substituting the genus-zero and elliptic determinant formulas into the two spectral relations, we
obtain closed explicit formulas for the determinants of the Bolza surface and the Klein quartic.
They are stated in Theorems~\ref{thm:bolza-final} and~\ref{thm:klein-final}.

The determinants $\det\Delta_B$ and $\det\Delta_K$ found here can themselves be used as smooth
reference determinants in the singular anomaly formula. The determinant of any singular metric on
$B$ or $K$, conformal to the corresponding hyperbolic metric, is then reduced to its global
conformal integral and local cone data. Whenever these data are explicit, this gives a further
closed determinant evaluation in genus two or three.

The representation-theoretic reduction remains valid under deformations preserving the topological
type of the group action. We apply it to the $V_4$-actions on the Bolza surface and the Klein
quartic, and obtain determinant and first-variation identities on equisymmetric strata of complex
dimensions $2$ and $3$, respectively. Independently of this reduction, we prove that every
compact quasiplatonic hyperbolic surface is a critical point of the spectral determinant on its
Teichm\"uller space.

This paper is organized as follows. Section~2 develops the virtual spectral reduction and collects
the explicit singular and elliptic input common to both calculations. Section~3 carries out the
reduction and evaluates the quotient determinants for the Bolza surface. Section~4 does the same for
the Klein quartic. Section~5 derives closed explicit values of the derivatives at $s=1$ of the
corresponding Selberg zeta functions and records the associated Artin factorization identities.
Section~6 extends the determinant identities to the equisymmetric deformation strata through the two
surfaces, derives the corresponding first-variation identities, and identifies their quasiplatonic
points. Section~7 proves the criticality theorem for compact quasiplatonic hyperbolic surfaces and
records its consequences for the quotient first variations. The appendix contains a sketch of an
alternative $V_4$-reduction of the determinant of the Klein quartic, providing an independent
check of the main calculation.

\section{Preliminaries}

\subsection{Spectral reduction}

Let a finite group $G$ act isometrically on a closed Riemannian manifold $X$. The quotients
below are understood as Riemannian orbifolds, and their Laplacians are the Friedrichs Laplacians.
For a subgroup $H\leq G$, let $\pi_H:X\to X/H$ be the quotient map. Pullback defines a unitary
map
$$
 U_H:L^2(X/H)\longrightarrow L^2(X)^H,
 \qquad U_Hu=|H|^{-1/2}u\circ\pi_H.
$$
Pullback also identifies the closed Dirichlet form on $X/H$ with the restriction to $L^2(X)^H$
of the closed Dirichlet form on $X$. The associated self-adjoint operators are the corresponding
Friedrichs Laplacians. Consequently,
$$
 U_H\Delta_{X/H}U_H^{-1}
 =\Delta_X\big|_{L^2(X)^H}.
$$
Thus the spectrum of $X/H$, counted with multiplicities, coincides with the spectrum of the
restriction of $\Delta_X$ to the $H$-invariant functions on $X$.

We shall use the following representation-theoretic form of Sunada's argument \cite{Sunada}; the
short proof is included for completeness.
\begin{lemma}\label{lem:artin}
Suppose that
\be\label{eq:representation-relation}
 \sum_H a_H\,\Ind_H^G\mathbf 1=0
\ee
as a virtual representation of $G$, where $a_H\in\mathbb Z$. Then the spectral zeta functions
$\zeta_{X/H}$ and the zeta-regularized determinants $\det\Delta_{X/H}$ satisfy
\be\label{eq:zeta-relation}
 \sum_H a_H\,\zeta_{X/H}(s)=0
\ee
and
\be\label{eq:det-relation-general}
 \prod_H\bigl(\det\Delta_{X/H}\bigr)^{a_H}=1.
\ee
\end{lemma}

\begin{proof}
Put $V_\lambda=\ker(\Delta_X-\lambda)$. Since the action of $G$ commutes with $\Delta_X$, the
eigenspace $V_\lambda$ is a finite-dimensional $G$-representation. Frobenius reciprocity gives
$$
 \dim\ker(\Delta_{X/H}-\lambda)
 =\dim V_\lambda^H
 =\left\langle V_\lambda,\Ind_H^G\mathbf 1
  \right\rangle_G,
$$
where $\langle\cdot,\cdot\rangle_G$ is the usual inner product of characters of $G$. Taking the
inner product of \eqref{eq:representation-relation} with $V_\lambda$, we obtain the corresponding
identity of multiplicities for every eigenvalue. This implies \eqref{eq:zeta-relation}.
Differentiating its meromorphic continuation at $s=0$, we obtain \eqref{eq:det-relation-general}.
The same multiplicity identity holds for the zero eigenvalue, which is omitted from the determinants
as in \eqref{ZetaDet}.
\end{proof}

In what follows, we choose the group and the permutation-representation relation to suit the
quotient geometry. For the Bolza surface we use a Klein four-group, whereas for the Klein quartic we
use a relation in $PSL(2,7)$. The elementary $V_4$-relation reappears in Section~6 on the
equisymmetric components through both surfaces, where it involves only quotient orbifolds of genera
zero and one.

Relations of this kind have a classical algebraic counterpart. Idempotent relations in
$\mathbb Q[G]$ give isogenies between Jacobians of quotient curves \cite{KaniRosen}. In
particular, the $V_4$-relation \eqref{eq:v4-relation} underlies the decomposition of the Jacobian
of the Bolza curve into two elliptic factors; an explicit group-algebra decomposition is given in
\cite{Jimenez}. For the Klein quartic, the corresponding algebraic statement
$\operatorname{Jac}(K)\sim T_K^3$, where $T_K$ has CM discriminant $-7$, together with an
explicit map $K\to T_K$, is given in \cite{FiteLorenzoGarciaSutherland}.

Let us stress that these algebraic decompositions concern only the action on first cohomology and
the period lattice. They give neither relations between the full spectra of the quotient Laplacians
nor their spectral determinants. Lemma~\ref{lem:artin}, on the other hand, applies the
permutation-representation relation to every eigenspace of $\Delta_X$. The resulting determinant
identities become explicit after the singular quotient determinants are evaluated.

\subsection{Singular anomaly formula}

The second analytic input used below is the singular Polyakov anomaly
formula~\cite{KalvinJFA}. 

 Let \(m_0\) be a smooth conformal metric on a closed Riemann
surface \(M\), and let \(m=e^{2\sigma}m_0\) be a conformal metric with
possible cone singularities at distinct points \(P_1,\ldots,P_n\).
In a local holomorphic coordinate \(z\) centered at \(P_j\), write
\[
 m=|z|^{2\beta_j}e^{2u_j(z)}|dz|^2,
 \qquad
 m_0=e^{2v_j(z)}|dz|^2,
 \qquad \beta_j>-1.
\]
For \(\beta_j\ne0\), the point \(P_j\) is a cone singularity of order
\(\beta_j\) and total angle \(2\pi(\beta_j+1)\); the value
\(\beta_j=0\) corresponds to a regular point.  The cone data are
encoded by the divisor
$\boldsymbol\beta=\sum_{j=1}^n\beta_j\cdot P_j$ of degree $|\boldsymbol\beta|=\sum_{j=1}^n\beta_j$, 
and \(m\) is said to represent \(\boldsymbol\beta\).

Let \(\Delta_m\) denote the Friedrichs Laplacian of the singular
metric \(m\), and let \(\Delta_{m_0}\) denote the Laplace--Beltrami
operator of the smooth metric \(m_0\).  Their determinants are
understood as in \eqref{ZetaDet}.

Let $K$ be the regularized Gaussian curvature of $m$, with the delta masses at the cone points $P_j$
omitted, and let $K_0$ be the Gaussian curvature of $m_0$. Denote the corresponding area forms
by $dA$ and $dA_0$. The anomaly formula from~\cite[Theorem~1.1]{KalvinJFA} reads
\be\label{eq:conical-anomaly}
\ba
 \log\frac{\det\Delta_m/\Area(M,m)}
 {\det\Delta_{m_0}/\Area(M,m_0)}
={}&-\frac{1}{12\pi}
 \left(\int_M K\sigma\,dA+\int_M K_0\sigma\,dA_0\right)
\\
&+\frac16\sum_{j=1}^n\beta_j
 \left(\frac{u_j(0)}{\beta_j+1}-v_j(0)\right)
 -\sum_{j=1}^n C(\beta_j).
\ea
\ee
Here the universal local contribution of a cone point of order $\beta>-1$ is
\be\label{eq:conical-constant}
 C(\beta)=
 2\zeta'_B(0;\beta+1,1,1)-2\zeta'_R(-1)
 -\frac{\beta^2}{6(\beta+1)}\log2-\frac{\beta}{12}
 +\frac12\log(\beta+1),
\ee
where $\zeta_B$ and $\zeta_R$ are the Barnes double and Riemann zeta functions, respectively.
The regularity conditions on \(u_j\) at \(z=0\) required in \cite{KalvinJFA} hold automatically for the constant-curvature metrics considered throughout this paper.

In particular, if a metric $m$ is multiplied by a constant $c>0$, then
$$
 \Delta_{cm}=c^{-1}\Delta_m,
 \qquad
 \zeta_{cm}(s)=c^s\zeta_m(s),
$$
and hence
\be\label{eq:determinant-scaling}
 \log\det\Delta_{cm}
 =\log\det\Delta_m-\zeta_m(0)\log c.
\ee
For a metric $m$ representing a divisor
$\boldsymbol\beta$, the value entering this formula is
\be\label{eq:zeta-zero}
 \zeta_m(0)
 =\frac{\chi_{\boldsymbol\beta}(M)}6
 -\frac1{12}\sum_{j=1}^n
 \left(\beta_j+1-\frac1{\beta_j+1}\right)-1,
 \qquad
 \chi_{\boldsymbol\beta}(M)
 =\chi(M)+|\boldsymbol\beta|,
\ee
where $\chi(M)$ is the topological  Euler characteristic of $M$. We shall use \eqref{eq:determinant-scaling}--\eqref{eq:zeta-zero} below when rescaling hyperbolic metrics to metrics of curvature $-1$.

\subsection{Three-cone sphere}

In this subsection we recall the results for the hyperbolic three-cone sphere that will be used
throughout the paper.

Let $m_{\boldsymbol\beta}=e^{2\varphi}|dt|^2$ be the unique unit-area hyperbolic metric on
$\mathbb P^1$ representing the divisor
$$
\boldsymbol\beta= \beta_0\cdot0+\beta_1\cdot1+\beta_\infty\cdot\infty
$$
of degree $|\boldsymbol\beta|= \beta_0+\beta_1+\beta_\infty<-2$. The cone
points are located at $0$, $1$, and $\infty$; this can always be achieved by a M\"obius transformation.

In the customary orbifold notation, used already in the Introduction, a three-cone sphere with cone
angles $2\pi/p$, $2\pi/q$, and $2\pi/r$ is denoted by $\mathbb P^1(p,q,r)$. In terms of the
divisor $\boldsymbol\beta$, the corresponding cone orders are $1/p-1$, $1/q-1$, and $1/r-1$.

The sphere $(\mathbb P^1,m_{\boldsymbol\beta})$ is isometric to the double of the hyperbolic
triangle with angles $\pi(\beta_0+1)$, $\pi(\beta_1+1)$, and $\pi(\beta_\infty+1)$. The metric
potential $\varphi$ satisfies the three-cone Liouville equation
\be\label{eq:three-cone-liouville}
 e^{-2\varphi}\bigl(-4\partial_t\partial_{\bar t}\varphi\bigr)
 =2\pi\bigl(|\boldsymbol\beta|+2\bigr),
 \qquad t\in\mathbb C\setminus\{0,1\}.
\ee
The Liouville equation \eqref{eq:three-cone-liouville} can be solved explicitly. Namely, the metric
$m_{\boldsymbol\beta}$ is the pullback of the model hyperbolic metric
$$
4\bigl(1+2\pi(|\boldsymbol\beta|+2)|w|^2\bigr)^{-2} {|dw|^2}
$$
by an appropriately normalized Schwarz triangle function. For details, we refer to
\cite[Appendix~A]{KalvinCV} and \cite[Section~2.1]{KalvinASNS}. In particular, the metric potential
has the asymptotics
\begin{align}
 \varphi(t)
 &=\beta_0\log|t|
   +\Psi(\beta_0,\beta_1,\beta_\infty)+o(1),
 &&t\to0,\notag\\
 \varphi(t)
 &=\beta_1\log|t-1|
   +\Psi(\beta_1,\beta_0,\beta_\infty)+o(1),
 &&t\to1,\label{eq:three-cone-regular-parts}\\
 \varphi(t)
 &=-(\beta_\infty+2)\log|t|
   +\Psi(\beta_\infty,\beta_1,\beta_0)+o(1),
 &&t\to\infty.\notag
\end{align}
The logarithmic terms encode the cone angles, and the regular parts are given by
\be\label{eq:Psi-definition}
\ba
\Psi(\beta_0,\beta_1,\beta_\infty)
={}&\log\frac{2\Gamma(-\beta_0)}{\Gamma(1+\beta_0)}+ \frac12\log
\frac{
 \Gamma\left(2+{|\boldsymbol\beta|}/{2}\right)
 \Gamma\left(\beta_0-{|\boldsymbol\beta|}/{2}\right)
}{
 \pi
 \Gamma\left(-\frac{|\boldsymbol\beta|}{2}\right)
 \Gamma\left(1+{|\boldsymbol\beta|}/{2}-\beta_0\right)}
\\
&+\frac12\log
\frac{
 \Gamma\left(1+{|\boldsymbol\beta|}/{2}-\beta_1\right)
 \Gamma\left(1+{|\boldsymbol\beta|}/{2}-\beta_\infty\right)
}{
 \Gamma\left(\beta_1-{|\boldsymbol\beta|}/{2}\right)
 \Gamma\left(\beta_\infty-{|\boldsymbol\beta|}/{2}\right)
}.
\ea
\ee
The $\log2$ discrepancy from the function $\Phi$ in \cite[Proposition~A.2]{KalvinCV} is due to a
different normalization of the cone-point coordinates; see \cite[Section~2.1]{KalvinASNS}.

For the three-cone sphere, we take the standard round metric on $\mathbb P^1$ as the smooth
reference metric $m_0$ in \eqref{eq:conical-anomaly}. In \cite{KalvinCV}, the global integral and
all local terms were evaluated explicitly, and the determinant of the Friedrichs Laplacian of
$m_{\boldsymbol\beta}$ was found in closed form; see \cite[Corollary~1.3]{KalvinCV}. Thus the
three-cone sphere is the basic nontrivial constant-curvature singular model for which both the
metric and its determinant are available in closed form.

\section{The Bolza surface}

\subsection{Geometry of the quotient orbifolds}\label{SubB_1}

Let $B$ be the Bolza surface equipped with its smooth metric of Gaussian curvature $-1$. Its
area is
$$
 \Area(B)=4\pi.
$$
Let $h$ be the hyperelliptic involution and let $\sigma$ be a nonhyperelliptic involution. Since
$h$ is central,
$$
 V=\{1,h,\sigma,h\sigma\}\cong C_2\times C_2.
$$
Put
$$
 T_1=B/\langle\sigma\rangle,\qquad
 T_2=B/\langle h\sigma\rangle,\qquad
 S_h=B/\langle h\rangle,\qquad
 S_V=B/V.
$$

The hyperelliptic involution has six fixed points. Each of $\sigma$ and $h\sigma$ has two fixed
points: by Riemann--Hurwitz, every involution on a genus-two curve other than the hyperelliptic
involution has a genus-one quotient and exactly two fixed points. The fixed-point sets of the three
nontrivial elements of $V$ are disjoint, since a finite stabilizer of an orientation-preserving
action at a point is cyclic.

It follows that
$$
 T_1,\ T_2:\quad
 \text{hyperbolic tori, each with two cone points of angle }\pi,
$$
and
$$
 S_h:\quad
 \text{hyperbolic sphere with six cone points of angle }\pi.
$$
For the genus $g(S_V)$, the Riemann--Hurwitz formula for the $V$-action gives
$$
 2
 =4\bigl(2g(S_V)-2\bigr)+(6+2+2),
$$
so $g(S_V)=0$. Since the fixed-point sets above are disjoint, each of the ten fixed points has
stabilizer of order two and hence belongs to a $V$-orbit of two points. They therefore give five
cone points on $S_V$: three from the fixed points of $h$, and one from the fixed points of each
of $\sigma$ and $h\sigma$. All five angles are $\pi$. Thus
$$
 S_V:\quad
 \text{hyperbolic sphere with five cone points of angle }\pi.
$$
The areas of the quotient metrics are
\ben
 \Area(T_1)=\Area(T_2)=\Area(S_h)=2\pi,
 \qquad
 \Area(S_V)=\pi.
\een
The two nonhyperelliptic involutions $\sigma$ and $h\sigma$ are conjugate in the conformal
automorphism group of the Bolza curve. Thus $T_1$ and $T_2$ are isometric.

\subsection{Reduction to quotient determinants}

We apply Lemma~\ref{lem:artin} to the action of the Klein four-group $V$ on the Bolza surface
$B$ and obtain the following determinant identity.

\begin{proposition}\label{prop:bolza-determinant}
The determinant $\det\Delta_B$ of the smooth hyperbolic Bolza surface $B$ satisfies
\be\label{eq:bolza-determinant-symmetric}
 \det\Delta_B
 =
 \frac{
  \bigl(\det\Delta_{T_B}\bigr)^2
  \det\Delta_{S_h}
 }{
 \bigl(\det\Delta_{S_V}\bigr)^2
 },
\ee
where $T_B$ denotes either of the isometric singular hyperbolic tori $T_1$ and $T_2$. The tori
$T_1,T_2$ and the spheres $S_h,S_V$ are those introduced in Subsection~\ref{SubB_1}.
\end{proposition}
\begin{proof}
Let $H_1,H_2,H_3$ be the three subgroups of order two in $V$. The following identity of
$V$-representations is elementary:
\be\label{eq:v4-relation}
 \Ind_{H_1}^{V}\mathbf 1+
 \Ind_{H_2}^{V}\mathbf 1+
 \Ind_{H_3}^{V}\mathbf 1
 =
 \Reg_V+2\mathbf 1_V.
\ee
Indeed, both sides have character six at the identity. At a nontrivial element $v\in V$, exactly
one of the three subgroups $H_j$ contains $v$; hence the character on both sides equals two.

Applying Lemma~\ref{lem:artin} to \eqref{eq:v4-relation} gives
$$
 \zeta_{T_1}(s)+\zeta_{T_2}(s)+\zeta_{S_h}(s)
 =
 \zeta_B(s)+2\zeta_{S_V}(s).
$$
Differentiating at $s=0$, we obtain
\be\label{eq:bolza-determinant}
 \det\Delta_B
 =
 \frac{
  \det\Delta_{T_1}\,
  \det\Delta_{T_2}\,
  \det\Delta_{S_h}
 }{
  \bigl(\det\Delta_{S_V}\bigr)^2
 }.
\ee
Denote either of the isometric tori $T_1$ and $T_2$ by $T_B$. Then
\eqref{eq:bolza-determinant} becomes \eqref{eq:bolza-determinant-symmetric}.
\end{proof}

We now evaluate the three quotient determinants in \eqref{eq:bolza-determinant-symmetric}. For the
Bolza surface, this $V$-reduction is particularly convenient. Its two torus quotients are
isometric CM tori of discriminant $-8$, and the three distinct quotient metrics arise as explicit
Belyi pullbacks of the metric on the hyperbolic $(2,3,8)$ three-cone sphere.

\subsection{The torus quotient}

To evaluate the torus quotient, we first rewrite the singular anomaly formula in terms of two
functionals adapted to a flat reference metric. Let $T$ be a complex torus. A nonzero holomorphic
differential $\omega$ determines the smooth flat metric $m_0=|\omega|^2$. Any singular metric in
the same conformal class can be written as
$$
 m=e^{2\sigma}m_0.
$$
Suppose that $m$ represents the divisor
$$
 \boldsymbol\beta=\sum_{j=1}^n\beta_j\cdot p_j,
 \qquad \beta_j>-1.
$$
In the local flat coordinate $z_j$, normalized by $dz_j=\omega$ and centered at $p_j$, write
$$
 \sigma(z_j)=\beta_j\log|z_j|+\sigma_j+o(1).
$$
Thus $\sigma_j$ is the regular part of the conformal factor $\sigma$ at the cone point $p_j$.
The singular Euler characteristic associated with the divisor is
$$
 \chi_{\boldsymbol\beta}(T)
 =\chi(T)+|\boldsymbol\beta|
 =|\boldsymbol\beta|.
$$
If $K_m$ denotes the regularized Gaussian curvature of $m$, with the point masses at the cone
points omitted, the Gauss--Bonnet formula reads $\int_TK_m\,dA=2\pi\chi_{\boldsymbol\beta}(T)$.

For the flat reference metric $m_0=|\omega|^2$, the Liouville action entering the singular anomaly
formula is
\be\label{eq:torus-liouville-action}
 \mathcal S_T[\sigma]
 =
 \int_TK_m\sigma\,dA
 +2\pi\sum_{j=1}^n\beta_j\sigma_j.
\ee
The remaining local contribution is conveniently collected in the $\mathcal H$-functional:
\be\label{eq:torus-H-functional}
 \log\mathcal H_T[\sigma]
 =
 2\sum_{j=1}^n
 \left(
 \beta_j+1-\frac1{\beta_j+1}
 \right)\sigma_j.
\ee
With these definitions, the anomaly formula \eqref{eq:conical-anomaly} can equivalently be written as
\be\label{eq:torus-action-anomaly}
 \log
 \frac{\det\Delta_m/\Area(T,m)}
 {\det\Delta_{|\omega|^2}/\Area(T,|\omega|^2)}
 =
 -\frac1{12\pi}
 \left(
 \mathcal S_T[\sigma]-\pi\log\mathcal H_T[\sigma]
 \right)
 -\sum_{j=1}^nC(\beta_j).
\ee

Combinations of a Liouville action $\mathcal S$ with local uniformization terms collected in
$\mathcal H$ also occur in the study of moduli and deformation spaces; see
\cite{TakhtajanZograf,TakhtajanTeo,ParkTakhtajanTeo}. At the orbifold values $\beta_j=1/m_j-1$,
such combinations enter the local-index formulas for hyperbolic orbifold metrics in
\cite{TakhtajanZografOrbifold,TakhtajanZografOrbifoldCorrection}. Here, by contrast,
$\mathcal S_T$ and $\mathcal H_T$ are obtained directly from the singular anomaly formula, which
applies to singular metrics of arbitrary curvature and arbitrary cone orders $\beta_j>-1$. Thus
the hyperbolic orbifold formulas are a discrete constant-curvature specialization of the singular
anomaly formalism; they are not used as a basis for the determinant calculation below. Related
singular Liouville data occur in \cite{CanWiegmann}.

The equality~\eqref{eq:torus-action-anomaly} separates the determinant into the smooth reference term,
the universal local cone terms, the local functional $\mathcal H_T$, and a single global quantity,
the Liouville action $\mathcal S_T$. The $\mathcal H_T$-functional is read off directly from the
regular parts of the metric potential, as carried out in the proof of
Proposition~\ref{prop:bolza-torus-determinant}. For the quotient geometries below, a Belyi map
transfers $\mathcal S_T$ to a three-cone sphere, where it is evaluated explicitly. This is what
makes a closed determinant calculation possible.

We now specialize this construction to the Bolza quotient. The only torus determinant in
\eqref{eq:bolza-determinant-symmetric} is $\det\Delta_{T_B}$. A convenient model of its underlying
elliptic curve is
\ben
 T_B:\qquad Y^2=F(X):=X^3-4X^2+2X.
\een
On this model we take
$$
 \omega_B=\frac{dX}{2Y},
 \qquad
 m_B=e^{2\sigma_B}|\omega_B|^2.
$$
The quotient metric $m_B$ has curvature $-1$ and two cone points of angle $\pi$. The quotient
map
$$
 J:T_B\longrightarrow \mathbb P^1(2,3,8)
$$
is given in the target coordinate $t$ by
\be\label{eq:bolza-belyi-map}
 t=J(X)=
 \frac{A(X)^3}{27F(X)^4},\qquad
 A(X)=3X^4+8X^3-20X^2+16X-4.
\ee

The target sphere $\mathbb P^1(2,3,8)$ represents the divisor
\be\label{eq:bolza-target-divisor}
 \boldsymbol\beta
 =-\frac23\cdot0
 -\frac12\cdot1
 -\frac78\cdot\infty.
\ee
Let $m_{\boldsymbol\beta}=e^{2\varphi}|dt|^2$ stand for the corresponding unit-area metric of
Gaussian curvature $-\pi/12$. On the target sphere $(\mathbb P^1,m_{\boldsymbol\beta})$, we
introduce the Liouville action
\be\label{eq:three-cone-liouville-action}
 \ba
 \mathcal S_{\boldsymbol\beta}[\varphi]
={}&2\pi\bigl(|\boldsymbol\beta|+2\bigr)
 \left(
 \int_{\mathbb P^1}\varphi\,dA_{\boldsymbol\beta}-1
 \right)\\
 &+2\pi\left[
 \beta_0\Psi(\beta_0,\beta_1,\beta_\infty)
 +\beta_1\Psi(\beta_1,\beta_0,\beta_\infty)\right.\\
 &\hspace{46mm}\left.
 +(\beta_\infty+2)\Psi(\beta_\infty,\beta_1,\beta_0)
 \right].
 \ea
\ee
Here the three $\Psi$-terms come from the regular parts in the
asymptotics~\eqref{eq:three-cone-regular-parts} of $\varphi$ at $0$, $1$, and $\infty$,
respectively. The expression \eqref{eq:three-cone-liouville-action} is formally different from the
regularized definitions of the Liouville action used by Zamolodchikov--Zamolodchikov and
Takhtajan--Zograf \cite{ZamolodchikovZamolodchikov,TakhtajanZograf}. Their equivalence is
established in \cite[Section~3]{KalvinCV}.

Since a sphere with three marked points has no moduli, $\mathcal S_{\boldsymbol\beta}[\varphi]$ is
a function only of the cone orders $\beta_0,\beta_1,\beta_\infty$. The
Zamolodchikov--Zamolodchikov variational relations express its derivatives with respect to these
orders through the regular parts of the metric potential. The regular parts are expressed by
$\Psi$ in \eqref{eq:three-cone-regular-parts} and determine the integration constant in the
variational equations. Their complete integration was carried out in \cite[Theorem~1.2]{KalvinCV}.
The passage from the normalization $(-1,0,1)$ used there to arbitrary positions of the marked
points is described in \cite[Remark~3.5]{KalvinCV}. In the present normalization $(0,1,\infty)$,
definition \eqref{eq:three-cone-liouville-action} is the specialization of
\cite[equation~(4.10)]{KalvinASNS}. Thus $\mathcal S_{\boldsymbol\beta}[\varphi]$ is known in
closed explicit form. Let us stress that this evaluation is essential: without it, the anomaly
formula retains an undetermined global term and cannot produce closed explicit formulas for the
determinants.

The next lemma transfers the explicitly known three-cone action to the torus.

\begin{lemma}\label{lem:bolza-liouville-action}
Let $\mathcal S_{T_B}[\sigma_B]$ be the source-torus Liouville action
\eqref{eq:torus-liouville-action} of the metric $m_B$, and let
$\mathcal S_{\boldsymbol\beta}[\varphi]$ be the target-sphere Liouville action
\eqref{eq:three-cone-liouville-action} of the unit-area hyperbolic metric $m_{\boldsymbol\beta}$
representing the divisor \eqref{eq:bolza-target-divisor}. Then
\be\label{eq:bolza-torus-action}
 \mathcal S_{T_B}[\sigma_B]
 =24\mathcal S_{\boldsymbol\beta}[\varphi]
 -\frac{22\pi}{3}\log2
 +\frac{\pi}{4}\log3-2\pi\log\pi-2\pi.
\ee
\end{lemma}

\begin{proof}
Let $p_\pm$ denote the two cone points of $m_B$. In the local flat coordinates $z_\pm$,
normalized by $dz_\pm=\omega_B$ and centered at $p_\pm$, write
$$
 \sigma_B(z_\pm)
 =-\frac12\log|z_\pm|+\sigma_{B,\pm}+o(1).
$$
The source divisor is
$$
 -\frac12\cdot p_+-\frac12\cdot p_-.
$$
Since $m_B$ has Gaussian curvature $-1$ and both source cone orders are $-1/2$, the Liouville
action \eqref{eq:torus-liouville-action} takes the form
\be\label{eq:bolza-source-action}
 \mathcal S_{T_B}[\sigma_B]
 =-\int_{T_B}\sigma_B\,dA_B
 -\pi(\sigma_{B,+}+\sigma_{B,-}).
\ee
We compute the global integral and the two regular parts in \eqref{eq:bolza-source-action} from the
pullback representation
\be\label{eq:bolza-pullback-identities}
 \ba
 m_B&=\frac{\pi}{12}J^*m_{\boldsymbol\beta}
 =e^{2\sigma_B}|\omega_B|^2,\\
 \sigma_B&=\varphi\circ J+
 \log\left|\frac{dJ}{\omega_B}\right|
 +\frac12\log\frac{\pi}{12}.
 \ea
\ee

The map $J$ in \eqref{eq:bolza-belyi-map} has degree $24$ and passport
$[\,3^8;\,1^2 2^{11};\,8^3\,]$. For example, the ramification over $1$ follows from
\be\label{eq:bolza-belyi-identity}
 A(X)^3-27F(X)^4
 =(2X-1)
 \bigl(8-40X+72X^2-56X^3+34X^4-18X^5\bigr)^2.
\ee
The two unramified points over $1$ are
\be\label{eq:bolza-conical-points}
 p_\pm=\left(\frac12,\pm\frac1{2\sqrt2}\right).
\ee
At every ramified point over $0$, $1$, or $\infty$, the local degree of $J$ cancels the
corresponding cone order. The only cone points of the pullback metric are the two unramified points
$p_\pm$ over $1$, both of angle $\pi$.

Since $|\boldsymbol\beta|+2=-1/{24}$, solving \eqref{eq:three-cone-liouville-action} for the
target integral gives
\be\label{eq:bolza-target-integral}
 \ba
 \int_{\mathbb P^1}\varphi\,dA_{\boldsymbol\beta}
 ={}&1-\frac{12}{\pi}\mathcal S_{\boldsymbol\beta}[\varphi]\\
 &-16\Psi\left(-\frac23,-\frac12,-\frac78\right)
 -12\Psi\left(-\frac12,-\frac23,-\frac78\right)\\
 &+27\Psi\left(-\frac78,-\frac12,-\frac23\right).
 \ea
\ee
Thus the target integral is explicit. Recall that $\mathcal S_{\boldsymbol\beta}[\varphi]$ was
found in closed form in \cite{KalvinCV}, and $\Psi$ is defined by \eqref{eq:Psi-definition}.

Set
$$
 f_B=\frac{dJ}{\omega_B}.
$$
Since the holomorphic differential $\omega_B$ has no zeros on $T_B$, $f_B$ is a meromorphic
function. In a local flat coordinate $z$, chosen so that $dz=\omega_B$, one has
$$
 f_B=\frac{dJ}{dz}=2YJ'(X).
$$
At the two points in \eqref{eq:bolza-conical-points},
\be\label{eq:bolza-local-derivative}
 |f_B(p_\pm)|=48\sqrt2.
\ee

We next introduce the norm used to transfer the logarithmic integral from $T_B$ to the target
sphere. For a regular value $t$ of $J$, the fibre $J^{-1}(t)$ consists of $24$ distinct
points. The norm of $f_B$ along $J$, denoted by $N_J(f_B)$, is defined by
$$
 N_J(f_B)(t)=\prod_{x\in J^{-1}(t)}f_B(x).
$$
Analytic continuation around a branch value only permutes the $24$ factors. Hence their product is
a single-valued meromorphic function of $t$, extending across the branch values.

The passport of $J$ determines the divisor of $N_J(f_B)$. At each of the eight points over
$0$, the map is locally cubic. Thus, in a local coordinate $z$,
$$
 J(z)=\lambda z^3+O(z^4),
 \qquad
 f_B(z)=3\lambda z^2+O(z^3),
$$
and $f_B$ has a zero of order $2$. The norm therefore has a zero of order $8\cdot2=16$ at
$t=0$.

Over $1$, there are eleven points at which $J$ is locally quadratic and two unramified points.
Each quadratic point contributes a simple zero of $f_B$, whereas the unramified points make no
contribution. Hence the norm has a zero of order $11$ at $t=1$.

Finally, at each of the three points over $\infty$, the map $J$ has a pole of order $8$.
Locally,
$$
 J(z)=\lambda z^{-8}+O(z^{-7}),
 \qquad
 f_B(z)=-8\lambda z^{-9}+O(z^{-8}),
$$
so $f_B$ has a pole of order $9$. The norm therefore has a pole of order $3\cdot9=27$ at
infinity. Consequently,
$$
 N_J(f_B)(t)=c_Bt^{16}(t-1)^{11}
$$
for a nonzero constant $c_B$.

To determine $|c_B|$, let $a$ be a zero of $A$. At either point $(a,Y(a))$ over $X=a$, in
the flat coordinate $z$ centered at that point,
$$
 J(z)=
 \lambda_a z^3+O(z^4),
 \qquad
 \lambda_a=\frac{8A'(a)^3}{27Y(a)^5}.
$$
The product of $f_B=dJ/dz$ over the three local inverse branches is $27\lambda_a t^2+o(t^2)$.
Pairing the two points with $Y(a)$ and $-Y(a)$, and then multiplying over the four zeros of
$A$, gives
$$
 |c_B|
 =
 2^{24}
 \frac{
 \left|\prod_{A(a)=0}A'(a)\right|^6
 }{
 \left|\prod_{A(a)=0}F(a)\right|^5
 }.
$$
Since $A$ has leading coefficient $3$,
$$
 \prod_{A(a)=0}A'(a)
 =\frac{\operatorname{Res}(A,A')}{3^3},
 \qquad
 \prod_{A(a)=0}F(a)
 =\frac{\operatorname{Res}(A,F)}{3^3}.
$$
Using
$$
 \operatorname{Res}(A,A')=-2^{18}3^2,
 \qquad
 \operatorname{Res}(A,F)=-2^{10},
$$
we obtain
$$
 |c_B|=2^{82}3^9.
$$
Therefore
\be\label{eq:bolza-norm}
 |N_J(f_B)(t)|
 =2^{82}3^9\,|t|^{16}|t-1|^{11}.
\ee

Multiplying the Liouville equation \eqref{eq:three-cone-liouville} for $m_{\boldsymbol\beta}$ by
$\log|t-\xi|$, $\xi\in\{0,1\}$, and applying Green's identity on the sphere with small disks
about $\xi$ and $\infty$ removed expresses the integral in terms of the corresponding regular
parts of $\varphi$. Since the curvature is $-\pi/12$, this implies 
\be\label{eq:bolza-logarithmic-integrals}
 \begin{split}
 \int_{\mathbb P^1}\log|t|\,dA_{\boldsymbol\beta}
 &=24\left[
 \Psi\left(-\frac23,-\frac12,-\frac78\right)
 -\Psi\left(-\frac78,-\frac12,-\frac23\right)
 \right],\\
 \int_{\mathbb P^1}\log|t-1|\,dA_{\boldsymbol\beta}
 &=24\left[
 \Psi\left(-\frac12,-\frac23,-\frac78\right)
 -\Psi\left(-\frac78,-\frac12,-\frac23\right)
 \right].
 \end{split}
\ee
The first identity in \eqref{eq:bolza-pullback-identities} gives
$$
 dA_B=\frac{\pi}{12}J^*dA_{\boldsymbol\beta}.
$$
Since $J$ has degree $24$, the two nonconstant terms in the second identity in
\eqref{eq:bolza-pullback-identities} satisfy
$$
 \ba
 \int_{T_B}\varphi\circ J\,dA_B
 &=\frac{\pi}{12}\,24
 \int_{\mathbb P^1}\varphi\,dA_{\boldsymbol\beta},\\
 \int_{T_B}\log|f_B|\,dA_B
 &=\frac{\pi}{12}
 \int_{\mathbb P^1}\log|N_J(f_B)|\,dA_{\boldsymbol\beta}.
 \ea
$$
The second identity is precisely the reason for introducing the norm: the logarithm of the product
over a fibre is the sum of the logarithms. The constant term in the second identity in
\eqref{eq:bolza-pullback-identities} contributes $\pi\log(\pi/12)$, since $\Area(T_B,m_B)=2\pi$.
Integrating the second identity in \eqref{eq:bolza-pullback-identities} and then substituting
\eqref{eq:bolza-norm} and \eqref{eq:bolza-logarithmic-integrals} gives
\be\label{eq:bolza-global-integral}
 \begin{split}
 \int_{T_B}\sigma_B\,dA_B
={}&\frac \pi {12}\left(
 24\int_{\mathbb P^1}\varphi\,dA_{\boldsymbol\beta}
 +\log(2^{82}3^9)+12\log\frac \pi {12}\right)\\
 &+2\pi\left\{
 16\left[
 \Psi\left(-\frac23,-\frac12,-\frac78\right)
 -\Psi\left(-\frac78,-\frac12,-\frac23\right)
 \right]\right.\\
 &\hspace{21mm}\left.
 +11\left[
 \Psi\left(-\frac12,-\frac23,-\frac78\right)
 -\Psi\left(-\frac78,-\frac12,-\frac23\right)
 \right]\right\}.
 \end{split}
\ee
Thus the global integral on the torus is expressed entirely through the already known three-cone
data.

By \eqref{eq:bolza-local-derivative}, the two regular parts of the potential of $m_B$ in the flat
coordinates $z_\pm$ coincide and are given by
\be\label{eq:bolza-source-regular-part}
 \sigma_{B,+}=\sigma_{B,-}
 =\Psi\left(-\frac12,-\frac23,-\frac78\right)
 +\frac12\log\frac \pi {12}
 +\frac12\log(48\sqrt2).
\ee
By \eqref{eq:bolza-target-integral}, the remaining sphere integral in
\eqref{eq:bolza-global-integral} is expressed through the target-sphere Liouville action. Hence
\eqref{eq:bolza-global-integral} gives the global term in the source-torus action, whereas
\eqref{eq:bolza-source-regular-part} gives its two local terms. Substitution into
\eqref{eq:bolza-source-action} and simplification yield \eqref{eq:bolza-torus-action}.
\end{proof}

\begin{proposition}\label{prop:bolza-torus-determinant}
For the hyperbolic torus $(T_B,m_B)$, whose two cone angles are $\pi$, the determinant of the
Friedrichs Laplacian is
\be\label{eq:bolza-torus-determinant}
 \begin{split}
 \log\det\Delta_{T_B}
 ={}&\log\frac{[\Gamma(1/8)\Gamma(3/8)]^2}
 {2^{289/72}3^{1/48}\pi^{25/12}}
 +\frac16-\frac2\pi\mathcal S_{\boldsymbol\beta}[\varphi]\\
 &-\frac12\Psi\left(-\frac12,-\frac23,-\frac78\right)
 -2C\left(-\frac12\right).
 \end{split}
\ee
Here $\mathcal S_{\boldsymbol\beta}[\varphi]$ denotes the Liouville action of the unit-area
hyperbolic metric $m_{\boldsymbol\beta}$ representing the divisor \eqref{eq:bolza-target-divisor}.
It was obtained in a  closed explicit  form in \cite[Theorem~1.2 and Remark~3.5]{KalvinCV}. The
functions $\Psi$ and $C$ are defined in \eqref{eq:Psi-definition} and
\eqref{eq:conical-constant}, respectively.
\end{proposition}

\begin{proof}
Since both cone orders are equal to $-1/2$, the definition \eqref{eq:torus-H-functional}
specializes to
\ben
 \log\mathcal H_{T_B}[\sigma_B]
 =-3(\sigma_{B,+}+\sigma_{B,-}).
\een
Substituting the two regular parts \eqref{eq:bolza-source-regular-part} gives
\be\label{eq:bolza-torus-H}
 \log\mathcal H_{T_B}[\sigma_B]
 =-6\Psi\left(-\frac12,-\frac23,-\frac78\right)
 -\frac{15}{2}\log2-3\log\pi.
\ee
Setting both cone orders equal to $-1/2$ in the anomaly formula
\eqref{eq:torus-action-anomaly} gives
\be\label{eq:bolza-source-action-anomaly}
 \log\frac{\det\Delta_{T_B}/(2\pi)}
 {\det\Delta_{|\omega_B|^2}/
 \Area(T_B,|\omega_B|^2)}
 =
 -\frac1{12\pi}
 \left(
 \mathcal S_{T_B}[\sigma_B]
 -\pi\log\mathcal H_{T_B}[\sigma_B]
 \right)
 -2C\left(-\frac12\right).
\ee

The remaining reference term in \eqref{eq:bolza-source-action-anomaly} is the normalized determinant
of the smooth flat metric $|\omega_B|^2$. Its period ratio is $\tau_B=i\sqrt2$. The Kronecker
limit formula \cite{OsgoodPhillipsSarnak} gives the first equality below, while the discriminant
$-8$ case of the Chowla--Selberg formula \cite{ChowlaSelberg,BarqueroSanchezEtAl}, together with
Euler's reflection formula, gives the second:
\be\label{eq:bolza-flat-determinant}
 \frac{\det\Delta_{|\omega_B|^2}}
 {\Area(T_B,|\omega_B|^2)}
 =\sqrt2\,|\eta(i\sqrt2)|^4
 =\frac{\bigl[\Gamma(1/8)\Gamma(3/8)\bigr]^2}
 {32\pi^3}.
\ee

All terms in the anomaly formula \eqref{eq:bolza-source-action-anomaly} are now explicit. The
source-torus Liouville action is evaluated in \eqref{eq:bolza-torus-action}, the local contribution
is \eqref{eq:bolza-torus-H}, and the normalized flat reference determinant is
\eqref{eq:bolza-flat-determinant}. Substitution and simplification yield
\eqref{eq:bolza-torus-determinant}.
\end{proof}

\subsection{The spherical quotients}

The two remaining terms in \eqref{eq:bolza-determinant-symmetric} are determinants of hyperbolic
spheres. We evaluate them by the Belyi-pullback formula of \cite{KalvinASNS}. Once the map to a
three-cone sphere is fixed, this formula expresses the determinant through the target determinant
and the local coefficients of the map.

\begin{lemma}\label{lem:bolza-octahedral-quotient}
Let $S_h$ be the regular octahedral sphere with six cone angles $\pi$, equipped with its metric
of Gaussian curvature $-1$, and let
$$
 \boldsymbol\gamma=
-\frac78\cdot0
-\frac12\cdot1
 -\frac78\cdot\infty.
$$
Then
\be\label{eq:bolza-octahedron-explicit}
 \begin{split}
 \log\det\Delta_{S_h}
={}&4\log\det\Delta_{\boldsymbol\gamma}
 +10\Psi\left(-\frac78,-\frac12,-\frac78\right)\\
&+8C\left(-\frac78\right)-2C\left(-\frac12\right)
-\frac12+12\zeta'_R(-1)
+\frac{47}{12}\log2+\frac{41}{12}\log\pi.
 \end{split}
\ee
Here $\det\Delta_{\boldsymbol\gamma}$ denotes the determinant of the Friedrichs Laplacian for the
unit-area hyperbolic metric on $\mathbb P^1$ representing $\boldsymbol\gamma$. The required
closed explicit formula was found in \cite[Corollary~1.3]{KalvinCV}. The functions $\Psi$ and
$C$ are defined in \eqref{eq:Psi-definition} and \eqref{eq:conical-constant}, respectively.
\end{lemma}

\begin{proof}
Specializing the octahedral formula \cite[Example~5.4, equation~(5.7)]{KalvinASNS} to the present
angles, we obtain the determinant for the metric of area $4\pi$ and Gaussian curvature $-1/2$.
The quotient metric on $S_h$ has area $2\pi$ and Gaussian curvature $-1$; thus it is obtained
by the constant rescaling $m\mapsto cm$, where $c=1/2$. By the rescaling property of determinants~\eqref{eq:determinant-scaling}, this
adds
$$
 -\zeta_{S_h}(0)\log c=-\frac5{12}\log2
$$
to the determinant, where $\zeta_{S_h}(0)=-5/12$ follows from the general formula~\eqref{eq:zeta-zero} for zeta  at zero. This leads to
\eqref{eq:bolza-octahedron-explicit} and completes the proof.
\end{proof}

\begin{lemma}\label{lem:bolza-five-cone-quotient}
Let $S_V$ be the hyperbolic sphere of area $\pi$ with five cone angles $\pi$, and let
$$
 \boldsymbol\beta=
 \left(-\frac23\right)\cdot0
 +\left(-\frac12\right)\cdot1
 +\left(-\frac78\right)\cdot\infty.
$$
Then
\be\label{eq:bolza-five-cone-explicit}
 \begin{split}
 \log\det\Delta_{S_V}
={}&\log12+12\log\det\Delta_{\boldsymbol\beta}
 +\frac{16}{3}\Psi\left(-\frac23,-\frac12,-\frac78\right)
 +\frac52\Psi\left(-\frac12,-\frac23,-\frac78\right)\\
 &+15\Psi\left(-\frac78,-\frac12,-\frac23\right)\\
&+12C\left(-\frac23\right)+7C\left(-\frac12\right)
 +12C\left(-\frac78\right)-\frac{11}{6}+44\zeta'_R(-1)\\
&+\frac{107}{9}\log2-\frac{11}{12}\log3
 +\frac{275}{24}\log\pi.
 \end{split}
\ee
Here $\det\Delta_{\boldsymbol\beta}$ denotes the determinant of the Friedrichs Laplacian for the
unit-area metric $m_{\boldsymbol\beta}$ on $\mathbb P^1(2,3,8)$ representing
$\boldsymbol\beta$. Its closed explicit formula was found in \cite[Corollary~1.3]{KalvinCV}. The
functions $\Psi$ and $C$ are defined in \eqref{eq:Psi-definition} and
\eqref{eq:conical-constant}, respectively.
\end{lemma}

\begin{proof}
Since $|\operatorname{Aut}(B)|=48$ and $|V|=4$, the natural map
$$
 S_V=B/V\longrightarrow B/\operatorname{Aut}(B)
 \simeq\mathbb P^1(2,3,8)
$$
has degree $12$. On $S_V\simeq\mathbb P^1_X$, the rational function
$$
 J_{S_V}(X)=\frac{A(X)^3}{27F(X)^4},
 \qquad
 \ba
 F(X)&=X^3-4X^2+2X,\\
 A(X)&=3X^4+8X^3-20X^2+16X-4,
 \ea
$$
has degree $12$ and passport
$$
 [\,3^4;\,1^2 2^5;\,4^3\,].
$$
Following the notation of \cite[Proposition~3.1, equation~(3.1)]{KalvinASNS}, write
$$
 J_{S_V}^{-1}\{0,1,\infty\}=\{x_1,\ldots,x_{14}\}.
$$
Denote the ramification order at $x_k$ by $\operatorname{ord}_kJ_{S_V}$, and let $c_k$ be the
first nonzero coefficient in the corresponding local expansion. Thus the local degree is
$\operatorname{ord}_kJ_{S_V}+1$, and, with $\xi_k=X-x_k$ for $x_k\ne\infty$ and $\xi_k=1/X$
for $x_k=\infty$,
$$
 J_{S_V}(X)-J_{S_V}(x_k)
 =c_k\xi_k^{\operatorname{ord}_kJ_{S_V}+1}
 +O\left(\xi_k^{\operatorname{ord}_kJ_{S_V}+2}\right)
$$
when $J_{S_V}(x_k)\ne\infty$, whereas at a pole
$$
 J_{S_V}(X)
 =c_k\xi_k^{-\operatorname{ord}_kJ_{S_V}-1}
 +O\left(\xi_k^{-\operatorname{ord}_kJ_{S_V}}\right).
$$

Set
$$
 Q(X)=8-40X+72X^2-56X^3+34X^4-18X^5.
$$
If $a$, $b$, and $q$ range over the roots of $A$, $F$, and $Q$, respectively, the
factorization \eqref{eq:bolza-belyi-identity} gives
$$
 \ba
 x_k=a:\quad
 &\operatorname{ord}_kJ_{S_V}=2,
 &c_k&=\frac{A'(a)^3}{27F(a)^4},
 &&A(a)=0,\\
 x_k=b:\quad
 &\operatorname{ord}_kJ_{S_V}=3,
 &c_k&=\frac{A(b)^3}{27F'(b)^4},
 &&F(b)=0,\\
 x_k=q:\quad
 &\operatorname{ord}_kJ_{S_V}=1,
 &c_k&=\frac{(2q-1)Q'(q)^2}{27F(q)^4},
 &&Q(q)=0,
 \ea
$$
together with
$$
 x_k=\frac12:\quad \operatorname{ord}_kJ_{S_V}=0,
 \quad c_k=96,
 \qquad
 x_k=\infty:\quad \operatorname{ord}_kJ_{S_V}=0,
 \quad c_k=24.
$$

The passport shows that $J_{S_V}^*m_{\boldsymbol\beta}$ has area $12$ and precisely five cone
angles $\pi$. The curvature $-1$ quotient metric is therefore
$$
 m_{S_V}=\frac{\pi}{12}J_{S_V}^*m_{\boldsymbol\beta}.
$$
The general formula for zeta at zero~\eqref{eq:zeta-zero} gives
$
 \zeta_{S_V}(0)=-\frac{11}{24}$.
So, by the rescaling property~\eqref{eq:determinant-scaling}, the rescaling factor $c=\pi/12$ contributes $\frac{11}{24}\log(\pi/12)$ into the determinant. Substituting the values of $\operatorname{ord}_kJ_{S_V}$ and
$c_k$ above into the Belyi-pullback formula of \cite{KalvinASNS} gives
\eqref{eq:bolza-five-cone-explicit}.
\end{proof}

\subsection{The determinant of the Bolza surface}

All three quotient determinants have now been evaluated explicitly. We can therefore reconstruct the
determinant of the Bolza surface.

\begin{theorem}\label{thm:bolza-final}
For the smooth metric of Gaussian curvature $-1$ on the Bolza surface, the spectral determinant
has the following closed explicit form:
\be\label{eq:bolza-final-expanded}
 \det\Delta_B=
 \frac{3^{17/24}(1+\sqrt2)^{1/4}}
 {2^{137/36}\pi^5}\,
 \mathcal G_B
 \exp\left\{
 -\frac12+\frac{68}{3}\zeta'_R(-1)+\mathcal Z_B
 \right\}.
\ee
Here $\mathcal G_B$ is the Gamma factor $$ \mathcal G_B= \Gamma\left(\frac14\right)^2
\Gamma\left(\frac18\right)^{7/2} \Gamma\left(\frac38\right)^{9/2} \left[
\frac{\displaystyle\prod_{r\in R_B^+}\Gamma(r/48)} {\displaystyle\prod_{r\in R_B^-}\Gamma(r/48)}
\right]^{3/4} $$ and $\mathcal Z_B$ is the Hurwitz zeta term $$ \ba \mathcal Z_B={}&
\frac{56}{3}\zeta'_H\left(-1,\frac12\right) +8\left[ \zeta'_H\left(-1,\frac13\right)
+\zeta'_H\left(-1,\frac23\right)\right]\\ &-\frac83\left[ \zeta'_H\left(-1,\frac14\right)
+\zeta'_H\left(-1,\frac34\right)\right] +\frac{28}{3}\left[ \zeta'_H\left(-1,\frac18\right)
+\zeta'_H\left(-1,\frac78\right)\right]\\ &-\frac43\left[ \zeta'_H\left(-1,\frac38\right)
+\zeta'_H\left(-1,\frac58\right)\right] -8\sum_{r\in R_B^-\cup R_B^+}\zeta'_H\left(-1,\frac
r{48}\right), \ea $$ where $R_B^-=\{1,17,25,41\}$ and $R_B^+=\{7,23,31,47\}$.
\end{theorem}

\begin{remark}
Numerical evaluation of the exact formula in Theorem~\ref{thm:bolza-final} gives
$$
 \det\Delta_B
 =4.7227328044455737935557601334449190323388\ldots.
$$
The first fourteen decimal places, $\det\Delta_B\approx4.72273280444557$, were previously obtained
by Strohmaier and Uski numerically from eigenvalues and the Selberg trace formula
\cite[Section~7.1]{StrohmaierUski}. Their computation therefore provides an independent numerical
check of our exact formula.
\end{remark}

\begin{proof}[Proof of Theorem~\ref{thm:bolza-final}]
Taking logarithms in the $V_4$-reduction \eqref{eq:bolza-determinant-symmetric} and substituting
\eqref{eq:bolza-torus-determinant}, \eqref{eq:bolza-octahedron-explicit}, and
\eqref{eq:bolza-five-cone-explicit}, we obtain \eqref{eq:bolza-final-expanded} after applying the
shift and reflection identities for the Hurwitz zeta and gamma functions. In this simplification all
universal local terms $C(\beta)$ cancel. The cancellation is structural: the cone points occur
only in the quotient problems, whereas the reconstructed Bolza metric is smooth.
\end{proof}

\section{The Klein quartic}

\subsection{Geometry of the quotient orbifolds}

Let $K$ be the Klein quartic with its smooth metric of Gaussian curvature $-1$. Its
orientation-preserving automorphism group is
$$
 G=PSL(2,7),\qquad |G|=168,
$$
and
$$
 \Area(K)=8\pi.
$$
Choose subgroups
$$
 C_3,\qquad C_7,\qquad
 F_{21}=C_7\rtimes C_3
$$
of orders $3$, $7$, and $21$, respectively, and set
$$
 T_K=K/C_3,\qquad
 S_{7,7,7}=K/C_7,\qquad
 S_{3,3,7}=K/F_{21}.
$$
The Hurwitz action of $G$ has signature $(2,3,7)$.

An order-three subgroup fixes two points of $K$. For the genus $g(T_K)$ of $T_K$,
Riemann--Hurwitz therefore gives
$$
 4=3\bigl(2g(T_K)-2\bigr)+2(3-1),
$$
so $g(T_K)=1$. The smooth hyperbolic metric on $K$ descends to a curvature $-1$ singular
metric on $T_K$, with the two fixed points of $C_3$ becoming its cone points. Thus
$$
 T_K:\quad
 \text{hyperbolic torus with two cone points of angle }\frac{2\pi}{3}.
$$

An order-seven subgroup fixes three points. Consequently
$$
 S_{7,7,7}:\quad
 \text{hyperbolic sphere with three cone points of angle }\frac{2\pi}{7}.
$$
Thus $S_{7,7,7}$ is the double of a hyperbolic triangle with angles
$\frac{\pi}{7},\frac{\pi}{7},\frac{\pi}{7}$.

The group $F_{21}$ contains its normal subgroup $C_7$ and seven subgroups of order three. Its
quotient has one orbifold point of order seven and two orbifold points of order three:
$$
 S_{3,3,7}:\quad
 \text{hyperbolic sphere with three cone points of angles }
 \frac{2\pi}{3},\frac{2\pi}{3},\frac{2\pi}{7}.
$$
Equivalently, it is the double of a hyperbolic triangle with angles
$\frac{\pi}{3},\frac{\pi}{3},\frac{\pi}{7}$.

The same signatures follow immediately from the orbifold Riemann--Hurwitz formula. The areas of the
quotient metrics are
$$
 \Area(T_K)=\frac{8\pi}{3},\qquad
 \Area(S_{7,7,7})=\frac{8\pi}{7},\qquad
 \Area(S_{3,3,7})=\frac{8\pi}{21}.
$$
These quotients are particularly well suited to an explicit determinant calculation. The two
spherical quotients are three-cone spheres. The remaining quotient $T_K$ is a CM torus of
discriminant $-7$ and admits an explicit degree-seven Belyi map to the hyperbolic sphere
$S_{3,3,7}$.

\subsection{Reduction to quotient determinants}

We apply the spectral reduction of Lemma~\ref{lem:artin} to the action of $G$ on the Klein quartic
and obtain the following determinant identity.

\begin{proposition}\label{prop:klein-determinant}
The determinant $\det\Delta_K$ of the smooth hyperbolic Klein quartic $K$ satisfies
\be\label{eq:klein-determinant}
 \det\Delta_K
 =
 \frac{
  \bigl(\det\Delta_{T_K}\bigr)^3
  \det\Delta_{S_{7,7,7}}
 }{
  \bigl(\det\Delta_{S_{3,3,7}}\bigr)^3
 }.
\ee
\end{proposition}

\begin{proof}
For the subgroups chosen above, the following identity of $G$-representations holds:
\be\label{eq:klein-character-relation}
 3\,\Ind_{C_3}^{G}\mathbf 1+
 \Ind_{C_7}^{G}\mathbf 1
 =
 \Reg_G+
 3\,\Ind_{F_{21}}^{G}\mathbf 1.
\ee
The identity can be checked directly on the six conjugacy classes
$$
 1A,\qquad 2A,\qquad 3A,\qquad 4A,\qquad 7A,\qquad 7B.
$$
The relevant permutation characters are
$$
\begin{array}{c|rrrrrr}
 &1A&2A&3A&4A&7A&7B\\ \hline
 \Ind_{C_3}^{G}\mathbf 1
   &56&0&2&0&0&0\\
 \Ind_{C_7}^{G}\mathbf 1
   &24&0&0&0&3&3\\
 \Ind_{F_{21}}^{G}\mathbf 1
   &8&0&2&0&1&1\\
 \Reg_G
   &168&0&0&0&0&0
\end{array}
$$
Thus the two sides of \eqref{eq:klein-character-relation} have the same character on every conjugacy
class.

Applying Lemma~\ref{lem:artin} to \eqref{eq:klein-character-relation}, we obtain the exact
spectral-zeta identity
$$
 3\zeta_{T_K}(s)+\zeta_{S_{7,7,7}}(s)
 =
 \zeta_K(s)+3\zeta_{S_{3,3,7}}(s).
$$
Differentiating at $s=0$, we obtain \eqref{eq:klein-determinant}.
\end{proof}

After rescaling, the three-cone formula of \cite{KalvinCV} evaluates the two spherical determinants
in \eqref{eq:klein-determinant}. Thus only the torus determinant $\det\Delta_{T_K}$ remains to be
found.

\subsection{The torus quotient}

Hoshino and Nakamura \cite{HoshinoNakamura} give the following model of the torus quotient
$T_K=K/C_3$:
\be\label{eq:klein-elliptic-model}
 T_K:\qquad v^2=4u^3+21u^2+28u.
\ee
It is isomorphic to the curve
$$
 Y^2=X^3-35X-98
$$
by $X=4u+7$ and $Y=4v$. In particular,
$$
 j(T_K)=-15^3=-3375,\qquad
 T_K\simeq
 \mathbb C/\left(
 \mathbb Z+\frac{1+i\sqrt7}{2}\mathbb Z\right).
$$

For this model set
$$
 \omega_K=\frac{du}{v},
 \qquad
 m_K=e^{2\sigma_K}|\omega_K|^2.
$$
The quotient metric $m_K$ has curvature $-1$ and two cone points of angle $2\pi/3$.

On the model \eqref{eq:klein-elliptic-model}, define
\be\label{eq:klein-s-map}
 s(u,v)=\frac12\left(
 (u^2+7u+7)v+7u^3+35u^2+49u+16\right)
\ee
and
$$
 J=\frac {i\sqrt 3}9 \left(s- 3 e^{\pi i/3}\right).
$$
Then $J:T_K\to\mathbb P^1$ has degree $7$ and passport
$
 [\,3,3,1;\,3,3,1;\,7\,]$.
It is the natural map
$$
 K/C_3\longrightarrow K/F_{21}=S_{3,3,7}.
$$
The two unramified points over $0$ and $1$ are
$$
 p_+=(3 e^{\pi i/3}-1,5(1-3 e^{\pi i/3})),\qquad
 p_-=\overline{ p_+}.
$$
All other points over $0,1,\infty$ become regular after the $(3,3,7)$-metric is pulled back. The
points $p_{\pm}$ remain cone points, both with angle $2\pi/3$.

As for the Bolza surface, the map $J$ transfers the nonflat part of the calculation to the target
sphere. The singular hyperbolic metric on $T_K$ is the pullback of the three-cone metric below.
The Liouville equation and the norm of $f_K=dJ/\omega_K$ then express the global and local terms
in the anomaly formula through the target data. The reference determinant is that of the flat CM
torus.

Let
\be\label{eq:klein-target-divisor}
 \boldsymbol\beta=
 -\frac23\cdot0
 -\frac23\cdot1
 -\frac67\cdot\infty
\ee
and let
$$
 m_{\boldsymbol\beta}=e^{2\varphi}|dt|^2
$$
be the unit-area metric on the target sphere representing $\boldsymbol\beta$. Its Gaussian
curvature is $-8\pi/21$. Hence $(8\pi/21)m_{\boldsymbol\beta}$ is precisely the curvature $-1$
metric on $S_{3,3,7}$. By \eqref{eq:three-cone-regular-parts}, the regular parts of $\varphi$ at
$0$ and $1$ are both $\Psi(-2/3,-2/3,-6/7)$, while the regular part at $\infty$ is
$\Psi(-6/7,-2/3,-2/3)$. For this target sphere we use the Liouville action
$\mathcal S_{\boldsymbol\beta}[\varphi]$ defined by \eqref{eq:three-cone-liouville-action}. The
evaluation in \cite[Theorem~1.2 and Remark~3.5]{KalvinCV} gives its closed explicit form; the
regular parts entering the action are expressed by \eqref{eq:Psi-definition}. We now transfer this
action through $J$ and evaluate the local functional $\mathcal H_{T_K}$.

\begin{lemma}\label{lem:klein-liouville-action}
Let $\mathcal S_{T_K}[\sigma_K]$ be the source-torus Liouville action
\eqref{eq:torus-liouville-action} of $m_K$, and let $\mathcal S_{\boldsymbol\beta}[\varphi]$ be
the target-sphere Liouville action \eqref{eq:three-cone-liouville-action} of the unit-area
hyperbolic metric $m_{\boldsymbol\beta}$ representing the divisor \eqref{eq:klein-target-divisor}.
Then
\be\label{eq:klein-torus-action}
 \begin{split}
 \mathcal S_{T_K}[\sigma_K]
={}&7\mathcal S_{\boldsymbol\beta}[\varphi]
 -8\pi\log2+\frac{24\pi}{7}\log3\\
 &-\frac{20\pi}{9}\log7-\frac{8\pi}{3}\log\pi
 -\frac{8\pi}{3}.
 \end{split}
\ee
\end{lemma}

\begin{proof}
Let $z_{\pm}$ be a local flat coordinate centered at $p_{\pm}$ and normalized by
$dz_{\pm}=\omega_K$. Write
$$
 \sigma_K(z_{\pm})
 =-\frac23\log|z_{\pm}|+\sigma_{K,{\pm}}+o(1).
$$
The source divisor is
$$
 -\frac23\cdot p_+-\frac23\cdot p_-.
$$
Inserting $K_{m_K}=-1$ and the two cone orders $-2/3$ into the torus Liouville action
\eqref{eq:torus-liouville-action} gives
\be\label{eq:klein-source-action}
 \mathcal S_{T_K}[\sigma_K]
 =-\int_{T_K}\sigma_K\,dA_K
 -\frac{4\pi}{3}(\sigma_{K,+}+\sigma_{K,-}).
\ee
For the target divisor, $|\boldsymbol\beta|+2=-4/21$, and \eqref{eq:three-cone-liouville-action}
becomes
\be\label{eq:klein-target-action}
 \begin{split}
 \mathcal S_{\boldsymbol\beta}[\varphi]
={}&-\frac{8\pi}{21}
 \left(\int_{\mathbb P^1}\varphi\,dA_{\boldsymbol\beta}-1\right)\\
 &-\frac{8\pi}{3}
 \Psi\left(-\frac23,-\frac23,-\frac67\right)
 +\frac{16\pi}{7}
 \Psi\left(-\frac67,-\frac23,-\frac23\right).
 \end{split}
\ee
We express the global integral and the two regular parts in \eqref{eq:klein-source-action} through
these target data.

Set
$$
 f_K=\frac{dJ}{\omega_K}.
$$
In any local flat coordinate $z$ normalized by $dz=\omega_K$, one has $f_K=dJ/dz$. A direct
differentiation of \eqref{eq:klein-s-map} gives
$$
 \frac{ds}{dz}(p_+)=49(2-3 e^{\pi i/3}),
 \qquad
 \frac{ds}{dz}(p_-)=49(2-3 e^{-\pi i /3}).
$$
It follows that
\be\label{eq:klein-local-derivative}
 |f_K(p_{\pm})|
 =\frac{49\sqrt{21}}9.
\ee
For a regular value $t$, define
$$
 N_J(f_K)(t)=\prod_{x\in J^{-1}(t)}f_K(x)
$$
to be the norm of $f_K$ along $J$. Its divisor is determined by the passport of $J$. The two
cubic points over $0$ contribute a zero of total order $4$, and the same holds over $1$. At
the unique point over $\infty$, $J$ has a pole of order $7$, and $f_K=dJ/dz$ therefore has a
pole of order $8$. Hence
$$
 N_J(f_K)=c_Kt^4(t-1)^4
$$
for a nonzero constant $c_K$.

At infinity, a flat coordinate $z$, chosen so that $dz=\omega_K$, satisfies
$$
 u=z^{-2}+O(z^{-1}),
 \qquad
 v=-2z^{-3}+O(z^{-2}).
$$
Consequently,
$$
 s=-z^{-7}+O(z^{-6}),
 \qquad
 J=-\frac{i\sqrt 3}{9}z^{-7}+O(z^{-6}).
$$
Let $z_1,\ldots,z_7$ be the local inverse branches over $t$ near infinity. Then
$$
 \prod_{j=1}^7z_j=\frac{i\sqrt 3}{9t}(1+o(1)),
 \qquad
 f_K(z)=\frac{i7\sqrt 3}{9}z^{-8}+O(z^{-7}).
$$
Therefore
$$
 N_J(f_K)(t)
 =\left(\frac{i7\sqrt 3}{9}\right)^7
 \left(\frac{i\sqrt 3}{9t}\right)^{-8}(1+o(1))
 =7^7(-i3\sqrt 3)t^8(1+o(1)),
$$
and hence $c_K=-i7^73\sqrt 3$. Thus
\be\label{eq:klein-norm}
 |N_J(f_K)|
 =3\sqrt3\,7^7|t|^4|t-1|^4.
\ee

As in the Bolza calculation, the Liouville equation for $m_{\boldsymbol\beta}$ yields
\be\label{eq:klein-logarithmic-integrals}
 \begin{split}
 \int_{\mathbb P^1}\log|t|\,dA_{\boldsymbol\beta}
 &=\frac{21}{4}\left[
 \Psi\left(-\frac23,-\frac23,-\frac67\right)
 -\Psi\left(-\frac67,-\frac23,-\frac23\right)\right],\\
 \int_{\mathbb P^1}\log|t-1|\,dA_{\boldsymbol\beta}
 &=\frac{21}{4}\left[
 \Psi\left(-\frac23,-\frac23,-\frac67\right)
 -\Psi\left(-\frac67,-\frac23,-\frac23\right)\right].
 \end{split}
\ee
For the metric potential of 
$$
 m_K=\frac{8\pi}{21}J^*m_{\boldsymbol\beta}
 =e^{2\sigma_K}|\omega_K|^2
$$
we therefore have
$$
 \sigma_K=\varphi\circ J+
 \log|f_K|
 +\frac12\log\frac{8\pi}{21}.
$$
Since $J$ has degree $7$, the two nonconstant terms satisfy
$$
 \ba
 \int_{T_K}\varphi\circ J\,dA_K
 &=\frac{8\pi}{21}\,7
 \int_{\mathbb P^1}\varphi\,dA_{\boldsymbol\beta},\\
 \int_{T_K}\log|f_K|\,dA_K
 &=\frac{8\pi}{21}
 \int_{\mathbb P^1}\log|N_J(f_K)|\,dA_{\boldsymbol\beta}.
 \ea
$$
The second identity follows, as before, from the definition of $N_J(f_K)$. The constant term
contributes
$$
 \frac{4\pi}{3}\log\frac{8\pi}{21},
$$
since $\Area(T_K,m_K)=8\pi/3$. Substituting \eqref{eq:klein-norm} and
\eqref{eq:klein-logarithmic-integrals} therefore gives
\be\label{eq:klein-global-integral}
 \begin{split}
 \int_{T_K}\sigma_K\,dA_K
={}&\frac{8\pi}{21}\left(
 7\int_{\mathbb P^1}\varphi\,dA_{\boldsymbol\beta}
 +\log(3\sqrt3\,7^7)+\frac72\log\frac{8\pi}{21}\right)\\
 &+16\pi\left[
 \Psi\left(-\frac23,-\frac23,-\frac67\right)
 -\Psi\left(-\frac67,-\frac23,-\frac23\right)\right].
 \end{split}
\ee

Since $J$ is unramified at $p_{\pm}$ and the target cone order at $0$ and $1$ is $-2/3$,
the pullback relation and \eqref{eq:klein-local-derivative} give the regular parts of the potential
of $m_K$ at $p_{\pm}$, in the flat coordinates $z_\pm$, as
\be\label{eq:klein-source-regular-parts}
 \sigma_{K,\pm}
 =\Psi\left(-\frac23,-\frac23,-\frac67\right)
 +\frac12\log\frac{8\pi}{21}
 +\frac13\log\frac{49\sqrt{21}}9.
\ee
Substitution of \eqref{eq:klein-global-integral} and \eqref{eq:klein-source-regular-parts} into the
source action \eqref{eq:klein-source-action}, followed by the use of \eqref{eq:klein-target-action},
yields \eqref{eq:klein-torus-action}.
\end{proof}

\begin{proposition}\label{prop:klein-torus-determinant}
For the hyperbolic torus $(T_K,m_K)$, whose two cone angles are $2\pi/3$, the determinant of the
Friedrichs Laplacian is
\be\label{eq:klein-torus-determinant}
 \begin{split}
 \log\det\Delta_{T_K}
={}&\log\frac{
 \bigl[\Gamma(1/7)\Gamma(2/7)\Gamma(4/7)\bigr]^2}
 {2^{8/3}3^{25/63}7^{1/9}\pi^{29/9}}
 +\frac29-\frac{7}{12\pi}
 \mathcal S_{\boldsymbol\beta}[\varphi]\\
 &-\frac89\Psi\left(-\frac23,-\frac23,-\frac67\right)
 -2C\left(-\frac23\right).
 \end{split}
\ee
Here $\mathcal S_{\boldsymbol\beta}[\varphi]$ denotes the Liouville action of the unit-area
hyperbolic metric $m_{\boldsymbol\beta}$ representing the divisor \eqref{eq:klein-target-divisor}.
This action is evaluated in a  closed explicit form  in \cite[Theorem~1.2 and Remark~3.5]{KalvinCV}. The functions $\Psi$ and $C$ are defined in \eqref{eq:Psi-definition} and
\eqref{eq:conical-constant}, respectively.
\end{proposition}

\begin{proof}
Since both cone orders are equal to $-2/3$, the definition \eqref{eq:torus-H-functional}
specializes to
$$
 \log\mathcal H_{T_K}[\sigma_K]
 =-\frac{16}{3}(\sigma_{K,+}+\sigma_{K,-}).
$$
Substituting the espressions~\eqref{eq:klein-source-regular-parts} for $\sigma_{K,\pm}$,  we obtain
\be\label{eq:klein-torus-H}
 \begin{split}
 \log\mathcal H_{T_K}[\sigma_K]
={}&-\frac{32}{3}
 \Psi\left(-\frac23,-\frac23,-\frac67\right)
 -16\log2+\frac{32}{3}\log3\\
 &-\frac{32}{9}\log7-\frac{16}{3}\log\pi.
 \end{split}
\ee
For the two cone orders $-2/3$, the anomaly formula
\eqref{eq:torus-action-anomaly} can equivalently be written as
\be\label{eq:klein-source-anomaly}
 \begin{split}
 \log\frac{\det\Delta_{T_K}/(8\pi/3)}
 {\det\Delta_{|\omega_K|^2}/\Area(T_K,|\omega_K|^2)}
={}&-\frac1{12\pi}
 \left(
 \mathcal S_{T_K}[\sigma_K]
 -\pi\log\mathcal H_{T_K}[\sigma_K]
 \right)\\
 &-2C\left(-\frac23\right).
 \end{split}
\ee

The remaining reference term in \eqref{eq:klein-source-anomaly} is the normalized determinant of the
smooth flat metric $|\omega_K|^2$. Its period ratio is $\tau_K=(1+i\sqrt7)/2$. The Kronecker
limit formula \cite{OsgoodPhillipsSarnak} gives the first equality below, while the discriminant
$-7$ case of the Chowla--Selberg formula \cite{ChowlaSelberg} gives the second:
\be\label{eq:klein-flat-determinant}
 \begin{split}
 \frac{\det\Delta_{|\omega_K|^2}}
 {\Area(T_K,|\omega_K|^2)}
 &=\frac{\sqrt7}{2}
 \left|\eta\left(\frac{1+i\sqrt7}{2}\right)\right|^4\\
 &=\frac{\bigl[\Gamma(1/7)\Gamma(2/7)\Gamma(4/7)\bigr]^2}
 {32\pi^4}.
 \end{split}
\ee
All terms in \eqref{eq:klein-source-anomaly} are now explicit. The source-torus Liouville action is
evaluated in \eqref{eq:klein-torus-action}, the local contribution is \eqref{eq:klein-torus-H}, and
the normalized flat reference determinant is \eqref{eq:klein-flat-determinant}. Substitution and
simplification yield \eqref{eq:klein-torus-determinant}.
\end{proof}

\subsection{The spherical quotients}

The two remaining terms in \eqref{eq:klein-determinant} are determinants of hyperbolic three-cone
spheres. We evaluate them using the unit-area formula of \cite{KalvinCV} and the scaling law
\eqref{eq:determinant-scaling}.

\begin{lemma}\label{lem:klein-spherical-quotients}
Let $S_{3,3,7}$ and $S_{7,7,7}$ be the hyperbolic spheres of curvature $-1$ introduced above.
The corresponding divisors are
$$
 \ba
 \boldsymbol\beta
 &=-\frac23\cdot0-\frac23\cdot1-\frac67\cdot\infty,\\
 \boldsymbol\gamma
 &=-\frac67\cdot0-\frac67\cdot1-\frac67\cdot\infty.
 \ea
$$
Then
\begin{align}
 \log\det\Delta_{S_{3,3,7}}
 &=
 \log\det\Delta_{\boldsymbol\beta}
 +\frac1{63}\log\frac{8\pi}{21},
 \label{eq:klein-S337-explicit}\\
 \log\det\Delta_{S_{7,7,7}}
 &=
 \log\det\Delta_{\boldsymbol\gamma}
 -\frac{13}{21}\log\frac{8\pi}{7}.
 \label{eq:klein-S777-explicit}
\end{align}
Here $\det\Delta_{\boldsymbol\beta}$ and $\det\Delta_{\boldsymbol\gamma}$ are the Friedrichs
determinants of the corresponding unit-area three-cone metrics. Their closed explicit formulas were
found in \cite[Corollary~1.3]{KalvinCV}.
\end{lemma}

\begin{proof}
Formula~\eqref{eq:zeta-zero} gives
$$
 \zeta_{S_{3,3,7}}(0)=-\frac1{63},
 \qquad
 \zeta_{S_{7,7,7}}(0)=\frac{13}{21}.
$$
The corresponding curvature $-1$ metrics have areas $8\pi/21$ and $8\pi/7$. The scaling law
\eqref{eq:determinant-scaling}, applied with these two factors, yields the stated identities.
\end{proof}

\subsection{The determinant of the Klein quartic}

All three quotient determinants have now been evaluated explicitly. We can therefore reconstruct the
determinant of the Klein quartic.

\begin{theorem}\label{thm:klein-final}
For the smooth metric of Gaussian curvature $-1$ on the Klein quartic, the spectral determinant
has the following closed explicit form:
\be\label{eq:klein-final-expanded}
 \det\Delta_K=
 \frac{7\,3^{8/7}}{2^8\pi^9}\,
 \mathcal G_K
 \exp\left\{
 -1+\frac{62}{3}\zeta'_R(-1)+\mathcal Z_K
 \right\}.
\ee
Here $\mathcal G_K$ is the Gamma factor
$$
 \mathcal G_K=
 \Gamma^6\left(\frac17\right)
 \frac{
 \left[\Gamma(2/7)\Gamma(4/7)\right]^{54/7}}
 {\left[\Gamma(3/7)\Gamma(5/7)\right]^{12/7}}
\times
 \left[
 \frac{\Gamma(5/21)\Gamma(19/21)}
 {\Gamma(2/21)\Gamma(16/21)}
 \right]^{12/7}
$$
and $\mathcal Z_K$ is the Hurwitz zeta term
$$
 \ba
 \mathcal Z_K={}&
 12\left[
 \zeta'_H\left(-1,\frac13\right)
 +\zeta'_H\left(-1,\frac23\right)\right]
 +7\left[
 \zeta'_H\left(-1,\frac17\right)
 +\zeta'_H\left(-1,\frac67\right)\right]\\
&-13\left[
 \zeta'_H\left(-1,\frac37\right)
 +\zeta'_H\left(-1,\frac47\right)\right]
 -\frac13\left[
 \zeta'_H\left(-1,\frac27\right)
 +\zeta'_H\left(-1,\frac57\right)\right]\\
&-6\sum_{r\in \{2,5,16,19\}}\zeta'_H\left(-1,\frac r{21}\right).
 \ea
$$
\end{theorem}

\begin{proof}
Substituting \eqref{eq:klein-torus-determinant}, \eqref{eq:klein-S337-explicit}, and
\eqref{eq:klein-S777-explicit} into the logarithmic form of \eqref{eq:klein-determinant} gives
\eqref{eq:klein-final-expanded}. Here we use the shift and reflection identities for the Hurwitz
zeta and gamma functions. In this calculation all universal local terms $C(\beta)$ cancel. As in
the Bolza case, the cancellation is structural: the cone points occur only in the quotient problems,
whereas the reconstructed Klein metric is smooth.
\end{proof}

\section{Selberg zeta values}


\begin{corollary}\label{cor:selberg-special-values}
Let $Z_{\mathrm{Sel},B}(s)$ and $Z_{\mathrm{Sel},K}(s)$ denote the Selberg zeta functions of the
Bolza surface and the Klein quartic, respectively, in the normalization of
\cite{DHokerPhong,SarnakDeterminants}. Then
\begin{align}
 Z'_{\mathrm{Sel},B}(1)
 &=
 \frac{3^{17/24}(1+\sqrt2)^{1/4}}
 {2^{173/36}\pi^6}\,
 \mathcal G_B
 \exp\left\{
 \frac{56}{3}\zeta'_R(-1)+\mathcal Z_B
 \right\},
 \label{eq:bolza-selberg-value}\\
 Z'_{\mathrm{Sel},K}(1)
 &=
 \frac{7\,3^{8/7}}
 {2^{10}\pi^{11}}\,
 \mathcal G_K
 \exp\left\{
 \frac{38}{3}\zeta'_R(-1)+\mathcal Z_K
 \right\}.
 \label{eq:klein-selberg-value}
\end{align}
Thus we obtain closed explicit formulas for the leading coefficients at $s=1$ of the Selberg zeta
functions of these two arithmetic hyperbolic surfaces.
\end{corollary}

\begin{proof}
For a smooth compact hyperbolic surface $M$ of genus $g$, the determinant--Selberg-zeta identity
is
$$
 \det\Delta_M
 =
 \mathfrak c^{\,2g-2}Z'_{\mathrm{Sel},M}(1),
 \qquad
 \mathfrak c=
 \sqrt{2\pi}\,
 \exp\left\{2\zeta'_R(-1)-\frac14\right\}.
$$
Applying this identity with $g=2$ and $g=3$ to \eqref{eq:bolza-final-expanded} and
\eqref{eq:klein-final-expanded}, respectively, yields the stated formulas.
\end{proof}

\begin{remark}
The virtual permutation-representation relations used above also give identities of the full Selberg
zeta functions. Writing $Z_{\mathrm{Sel},Y}(s)$ for the Selberg zeta function of a hyperbolic
quotient orbifold $Y$, the Artin factorization formalism of \cite{VenkovZograf} gives
\begin{align}
 Z_{\mathrm{Sel},B}(s)\,
 Z_{\mathrm{Sel},S_V}(s)^2
 &=
 Z_{\mathrm{Sel},T_1}(s)\,
 Z_{\mathrm{Sel},T_2}(s)\,
 Z_{\mathrm{Sel},S_h}(s),
 \label{eq:bolza-selberg-factorization}\\
 Z_{\mathrm{Sel},K}(s)\,
 Z_{\mathrm{Sel},S_{3,3,7}}(s)^3
 &=
 Z_{\mathrm{Sel},T_K}(s)^3\,
 Z_{\mathrm{Sel},S_{7,7,7}}(s).
 \label{eq:klein-selberg-factorization}
\end{align}
The factorization identities themselves are classical consequences of the virtual representation
relations. Here we find their leading coefficients at $s=1$ for the Bolza surface and the Klein
quartic in closed explicit form.
\end{remark}

\section{Determinant identities on equisymmetric strata}

The representation-theoretic reduction is not confined to the two surfaces considered in Sections~3
and~4. A permutation-representation identity persists under every deformation that preserves the
topological type of the group action. It therefore yields an identity of determinant functions on
the corresponding equisymmetric Teichm\"uller space.

\begin{proposition}\label{prop:equisymmetric-family}
Let $\mathcal T$ be a connected equisymmetric Teichm\"uller space of closed hyperbolic surfaces
$X_t$ of genus $g\geq2$, equipped with a marked action of a finite group $G$. Suppose that
$$
 \sum_H a_H\,\Ind_H^G\mathbf1=0
$$
as a virtual representation of $G$. Then
\be\label{eq:equisymmetric-determinant}
 \prod_H
 \bigl(\det\Delta_{X_t/H}\bigr)^{a_H}=1,
\ee
and consequently
\be\label{eq:equisymmetric-variation}
 \sum_H a_H\,d\log\det\Delta_{X_t/H}=0
\ee
on $\mathcal T$.
\end{proposition}

\begin{proof}
Applying Lemma~\ref{lem:artin} to each surface $X_t$, we obtain
\eqref{eq:equisymmetric-determinant}. Its differential is \eqref{eq:equisymmetric-variation}.
\end{proof}

\subsection{The stratum through the Bolza surface}

Consider the equisymmetric component of genus-two surfaces carrying the $V_4$-action used in
Section~3. For a surface $B_t$ in this component, write
$$
 T_{1,t}=B_t/\langle\sigma\rangle,\qquad
 T_{2,t}=B_t/\langle h\sigma\rangle,\qquad
 S_{h,t}=B_t/\langle h\rangle,\qquad
 S_{V,t}=B_t/V_4.
$$
The quotient $S_{V,t}$ is a sphere with five cone points of angle $\pi$. The equisymmetric
component is therefore locally parametrized by the positions of five marked points on the sphere and
has complex dimension $2$.

\begin{corollary}\label{cor:bolza-stratum}
On this component,
\be\label{eq:bolza-stratum}
 \det\Delta_{B_t}
 =
 \frac{
 \det\Delta_{T_{1,t}}\,
 \det\Delta_{T_{2,t}}\,
 \det\Delta_{S_{h,t}}
 }{
 \bigl(\det\Delta_{S_{V,t}}\bigr)^2
 }.
\ee
In particular,
\be\label{eq:bolza-stratum-variation}
 \begin{split}
 d\log\det\Delta_{B_t}
 ={}&d\log\det\Delta_{T_{1,t}}
 +d\log\det\Delta_{T_{2,t}}
 +d\log\det\Delta_{S_{h,t}}\\
 &-2d\log\det\Delta_{S_{V,t}}.
 \end{split}
\ee
\end{corollary}

\begin{proof}
Apply Proposition~\ref{prop:equisymmetric-family} to the representation identity
\eqref{eq:v4-relation}. Equation~\eqref{eq:bolza-stratum-variation} is the differential of
\eqref{eq:bolza-stratum}.
\end{proof}

At the Bolza point the two elliptic quotients are isometric. Away from that point they need not be
isometric, but the determinant identity \eqref{eq:bolza-stratum} remains unchanged.

\subsection{The common \texorpdfstring{$D_8$}{D8}-locus}

The Bolza surface and the quasiplatonic curve
$$
 C_6:\qquad y^2=x^6-1
$$
lie on a one-dimensional sublocus of the $V_4$-component on which the symmetry group contains a
dihedral group $D_8$ of order $8$. The orbifold structure and real forms of this equisymmetric
locus were described in \cite{CostaRiera}, where the group of order $8$ is denoted by $D_4$. The
standard normal form for this locus \cite[Sections~2 and~3]{ShaskaVoelklein} is
\be\label{eq:d8-genus-two-family}
 X_t:\qquad
 y^2=x^6-tx^4+tx^2-1
 =(x^2-1)\bigl(x^4+(1-t)x^2+1\bigr),
 \qquad t\in\mathbb C\setminus\{-1,3\}.
\ee

\begin{proposition}\label{prop:d8-genus-two-locus}
The family \eqref{eq:d8-genus-two-family} joins the Bolza surface to the equianharmonic curve:
$$
 X_{-5}\simeq B,\qquad X_0=C_6.
$$
The $V_4$-action used in \eqref{eq:bolza-stratum} persists throughout the family, and its two
elliptic quotients are isometric. Their common $j$-invariant is
\be\label{eq:d8-elliptic-j}
 j_T(t)=\frac{256t^3}{t+1}.
\ee

The $D_8$-quotient $Z_t=X_t/D_8$ is the sphere with coordinate
\be\label{eq:d8-full-quotient-map}
 z=\frac{(x^2+1)^2}{4x^2}.
\ee
Its four cone points are
$$
 z=0,\quad z=1,\quad
 z=a(t):=\frac{t+1}{4},\quad z=\infty,
$$
with orbifold orders $2,4,2,2$, respectively.
\end{proposition}

\begin{proof}
The hyperelliptic involution
$$
 h(x,y)=(x,-y)
$$
and the involution
$$
 \sigma(x,y)=(-x,y)
$$
act on $X_t$. The curve also admits
$$
 \rho(x,y)=
 \left(\frac1x,\frac{iy}{x^3}\right).
$$
Indeed,
$$
 x^6\left(
 x^{-6}-tx^{-4}+tx^{-2}-1\right)
 =-\left(x^6-tx^4+tx^2-1\right).
$$
Moreover,
$$
 \rho^2=h,\qquad \sigma^2=1,\qquad
 \sigma\rho\sigma=\rho^{-1}.
$$
Thus $\rho$ and $\sigma$ generate $D_8$, and $\{1,h,\sigma,h\sigma\}$ is the distinguished
$V_4$. The right hand side of \eqref{eq:d8-genus-two-family} has multiple roots precisely for
$t=-1$ or $t=3$.

For $t=0$, equation \eqref{eq:d8-genus-two-family} is the equation of $C_6$. For $t=-5$, its
branch points are
$$
 \{\pm1,\ \pm i(\sqrt2-1),\
          \pm i(\sqrt2+1)\}.
$$
The M\"obius transformation
$$
 x\longmapsto
 e^{-\pi i/4}\frac{x-1}{x+1}
$$
maps this set onto $\{0,\infty,\pm1,\pm i\}$, the branch set of $y^2=x(x^4-1)$. Hence
$X_{-5}\simeq B$.

The involutions $\sigma$ and $h\sigma$ are conjugate by $\rho$, so their singular elliptic
quotients are isometric. With $u=x^2$, one of the underlying elliptic curves is
$$
 y^2=u^3-tu^2+tu-1.
$$
A direct computation gives \eqref{eq:d8-elliptic-j}. In particular,
$$
 j_T(-5)=8000,\qquad j_T(0)=0,
$$
in agreement with the two CM elliptic quotients used above.

The function in \eqref{eq:d8-full-quotient-map} is invariant under the reduced transformations
$x\mapsto-x$ and $x\mapsto1/x$. The points fixed by the reduced $V_4$-action map to
$0,1,\infty$. The two Weierstrass points $x=\pm1$ have stabilizer of order $4$ and map to
$z=1$. The remaining four Weierstrass points satisfy
$$
 x^2+x^{-2}=t-1
$$
and hence map to $z=(t+1)/4$; their stabilizer has order $2$. This proves the asserted
signature.
\end{proof}

Let $m_a$ be the curvature-$-1$ metric on $\mathbb P^1_z$ with cone angles
$\pi,\pi/2,\pi,\pi$ at $0,1,a,\infty$, respectively. The hyperbolic metrics on $X_t$ and on
all four intermediate quotients in \eqref{eq:bolza-stratum} are pullbacks of $m_{a(t)}$. At the
two distinguished points,
$$
 t=-5,\quad a=-1,
 \qquad\text{and}\qquad
 t=0,\quad a=\frac14.
$$
The additional automorphisms at these values produce compatible maps from the four-cone quotient to
the three-cone models of signatures $(0;2,3,8)$ and $(0;2,4,6)$, respectively. Thus
Proposition~\ref{prop:d8-genus-two-locus} identifies the common four-cone geometry behind the two
genus-two Belyi configurations.

For a general $t$, the determinant identity \eqref{eq:bolza-stratum} remains exact, and all
quotient metrics in this identity are pullbacks of the four-cone metric $m_{a(t)}$. Thus the
obstruction to an explicit evaluation lies entirely in the analytic data of this four-cone sphere,
not in the spectral reduction. As explained in \cite[Introduction and Section~2]{KalvinCV}, the
anomaly formula extends to an arbitrary number of cone points, whereas, starting with four points,
no closed explicit construction of the uniformizing metric, its regular parts, or its Liouville
action is known in general. In the present family these data depend on the single cross-ratio
$a=(t+1)/4$.

More precisely, the hyperbolic Liouville action generates both the accessory parameters and the
regular parts of the conformal factor: in the notation of \cite[Lemmas~4.4 and~4.5]{KalvinASNS},
$$
 -\frac{1}{2\pi}\,\partial_a S=h_a,
 \qquad
 \varphi_k=
 \frac12+\frac{1}{4\pi}\,\partial_{\beta_k}S.
$$
Consequently, an explicit four-point action $S(a;\boldsymbol\beta)$, or equivalently the
corresponding accessory and local data, would make all terms in the singular anomaly formula
explicit. The remaining flat elliptic determinant is determined by $j_T(t)$ through the Kronecker
limit formula. Together with the $V_4$-identity, this would determine $\det\Delta_{X_t}$ by
integrating the Polyakov one-form along the $D_8$-locus, with normalization at either $a=-1$ or
$a=1/4$. The Zamolodchikov--Zamolodchikov factorization of the four-point classical Liouville
action through classical conformal blocks provides a conjectural approach to this problem
\cite{ZamolodchikovZamolodchikov,HadaszJaskolski}; strong numerical evidence is available in the
parabolic four-point case \cite{HadaszJaskolskiPiatek}, but no corresponding rigorous formula is
presently known for the cone singularities occurring here.

\subsection{The stratum through the Klein quartic}

The full $PSL(2,7)$-action used in Section~4 is rigid and therefore does not produce a nontrivial
equisymmetric deformation family. We use instead a subgroup $V\simeq V_4$. Its spectral reduction
is also suitable for an explicit evaluation of $\det\Delta_K$, because the quotient orbifolds
admit Belyi maps to the $(2,3,7)$ sphere. Appendix~\ref{app:klein-v4-check} contains a sketch of
this independent calculation.

Denote the three subgroups of order two in $V\simeq V_4$ by $H_1,H_2,H_3$. The Hurwitz action on
the Klein quartic $K$ has $84$ points with stabilizer of order two, and $PSL(2,7)$ has $21$
involutions. Since the involutions are conjugate, each of them fixes four points. The fixed-point
sets of the three nontrivial elements of $V$ are disjoint, because a stabilizer of an
orientation-preserving action at a point is cyclic.

Riemann--Hurwitz gives
$$
 4=2\bigl(2g(K/H_j)-2\bigr)+4,
 \qquad
 4=4\bigl(2g(K/V)-2\bigr)+12.
$$
Thus the three quotients by $H_j$ have genus one and four cone points of angle $\pi$, whereas
the quotient by $V$ is a sphere with six such points. These branching data remain fixed on the
corresponding equisymmetric component. This component is locally parametrized by six marked points
on the sphere and has complex dimension $3$.

For a surface $K_t$ in this component, put
$$
 T_{j,t}=K_t/H_j,\qquad j=1,2,3,
 \qquad
 S_{K,t}=K_t/V.
$$

\begin{corollary}\label{cor:klein-stratum}
On the equisymmetric component through the Klein quartic,
\be\label{eq:klein-stratum}
 \det\Delta_{K_t}
 =
 \frac{
 \det\Delta_{T_{1,t}}\,
 \det\Delta_{T_{2,t}}\,
 \det\Delta_{T_{3,t}}
 }{
 \bigl(\det\Delta_{S_{K,t}}\bigr)^2
 },
\ee
and hence
\be\label{eq:klein-stratum-variation}
 d\log\det\Delta_{K_t}
 =
 \sum_{j=1}^3d\log\det\Delta_{T_{j,t}}
 -2d\log\det\Delta_{S_{K,t}}.
\ee
\end{corollary}

\begin{proof}
The representation identity \eqref{eq:v4-relation} applies to the marked action of $V$ on every
surface in the component. Proposition~\ref{prop:equisymmetric-family} yields
\eqref{eq:klein-stratum}; differentiation yields \eqref{eq:klein-stratum-variation}.
\end{proof}

\subsection{Quasiplatonic points in the two strata}

Here quasiplatonic is understood in its rigid sense: the quotient by the full conformal automorphism
group is a sphere with three cone points. This is substantially more restrictive than the mere
existence of a Belyi function. The latter holds at infinitely many algebraic points of either
stratum, whereas the quasiplatonic points in each fixed genus form a finite set. A Belyi map whose
ramification data do not reproduce the prescribed cone angles does not yield the quotient metric
required for the determinant calculation.

\begin{proposition}\label{prop:quasiplatonic-points}
The genus-two equisymmetric component considered above contains exactly two quasiplatonic points,
represented by
$$
 B:\ y^2=x^5-x,
 \qquad
 C_6:\ y^2=x^6-1.
$$
Their full automorphism quotients have signatures $(0;2,3,8)$ and $(0;2,4,6)$, respectively.

The genus-three component contains exactly five quasiplatonic points. In addition to the Klein and
Fermat quartics, put
$$
 \ba
 Q_{48}&:\quad
 x^4+y^4+z^4+(4\zeta_3+2)x^2y^2=0,
 \qquad \zeta_3=e^{2\pi i/3},\\
 H_{32}&:\quad y^2=x^8-1,\\
 H_{48}&:\quad y^2=x^8+14x^4+1.
 \ea
$$
The orders of the full conformal automorphism groups and the signatures of the corresponding
quotients are
$$
\begin{array}{ccc}
 \text{surface}
 &|\operatorname{Aut}|&\text{quotient signature}\\ \hline
 \text{Klein quartic}&168&(0;2,3,7)\\
 \text{Fermat quartic}&96&(0;2,3,8)\\
 Q_{48}&48&(0;2,3,12)\\
 H_{32}&32&(0;2,4,8)\\
 H_{48}&48&(0;2,4,6)
 \end{array}
$$
\end{proposition}

\begin{proof}
For a compact curve of genus at least two, the equisymmetric stratum associated with its full
automorphism group is zero-dimensional precisely when the full quotient is a three-cone sphere, that
is, when the curve is quasiplatonic. By \cite[Theorem~7.1]{SingermanSyddall}, there are exactly
three quasiplatonic surfaces of genus two, with full automorphism groups of orders $10$, $24$,
and $48$. The order-$10$ group is cyclic and contains no $V_4$; the other two surfaces are
$C_6$ and $B$, and both carry the distinguished $V_4$-action.

In genus three, the non-hyperelliptic zero-dimensional full-automorphism strata are represented by
the $C_9$-curve, $Q_{48}$, the Fermat quartic, and the Klein quartic
\cite[Table~1]{BergstromVanderGeer}. The first contains no $V_4$, whereas the other three contain
a sign-change $V_4$, or a conjugate, and hence belong to the present stratum by
\cite[Lemma~2]{BouwEtAl}. The hyperelliptic zero-dimensional strata have groups $C_{14}$, $U_6$,
$V_8$, and $C_2\times S_4$ \cite[Table~1]{GutierrezSevillaShaska}. Membership in the present
stratum requires a subgroup $C_2^3$ with a distinguished $V_4$ not containing the hyperelliptic
involution \cite[Lemma~1]{BouwEtAl}. This excludes $C_{14}$ and $U_6\simeq C_4\times S_3$,
leaving $H_{32}$ and $H_{48}$. The stated signatures follow from the corresponding triangle
actions and the Riemann--Hurwitz formula.
\end{proof}

In this paper we evaluate the Bolza surface and the Klein quartic. For $C_6$, identified over
$\mathbb C$ with the model $y^2=x^6+1$, the quotient by a non-hyperelliptic involution is an
elliptic curve with $j=0$; see \cite[Section~4 and Lemma~4.1]{FiteSutherland}. For the Fermat
quartic, the three sign-change quotients and their maps to the $V_4$-quotient are described
explicitly in \cite[Section~2.1]{BouwEtAl}. They are isomorphic elliptic curves with $j=1728$, in
agreement with the explicit CM elliptic factor and map in \cite[(3.1), (3.2), and
Proposition~3.1]{FiteLorenzoGarciaSutherland}. Thus the same $V_4$-relation and singular anomaly
formula can be used to evaluate the determinants at both points. This reduction also applies to
$Q_{48}$, $H_{32}$, and $H_{48}$ once the quotient maps associated with the distinguished
$V_4$-actions are written explicitly. We do not pursue these additional calculations here. Let us
also stress that the list is exhaustive only among quasiplatonic points, not among all points at
which compatible explicit quotient uniformizations may exist.

The equisymmetric deformations move the branch values on $S_{V,t}$ and $S_{K,t}$ while keeping
their orbifold orders and topological monodromy fixed. This changes the complex structure of the
covering surface while preserving the smoothness of its pulled-back hyperbolic metric.

\section{Critical points of the spectral determinant}

We consider $\log\det\Delta$ as a function on the Teichm\"uller space $\mathcal T_g$ of marked
compact Riemann surfaces of genus $g$, assigning to each conformal structure its unique metric of
curvature $-1$. The following assertion concerns the full Teichm\"uller space, not only an
equisymmetric stratum.

\begin{proposition}\label{prop:bolza-klein-critical}
Every compact quasiplatonic hyperbolic surface is a critical point of the spectral determinant
$\det\Delta$ on its Teichm\"uller space. In particular, this holds for the Bolza surface and the
Klein quartic.
\end{proposition}

\begin{proof}
Let $M$ be a compact quasiplatonic hyperbolic surface, let $G=\operatorname{Aut}(M)$, and put
$$
 F=\log\det\Delta.
$$
Since $F$ is invariant under pullback, its differential at $M$ is $G$-invariant. Hence its
$(1,0)$-part
$$
 \partial F_M\in T_M^{*(1,0)}\mathcal T_g
 \simeq H^0(M,K_M^2)
$$
corresponds to a $G$-invariant holomorphic quadratic differential $q$ on $M$.

The differential $q$ descends to a meromorphic quadratic differential $\widetilde q$ on the
orbifold quotient $M/G$, with at most simple poles at the orbifold points. Indeed, near a point
with stabilizer of order $m$, choose coordinates in which the quotient map is $w=z^m$, and write
$q=f(z)\,dz^2$. Invariance under $z\mapsto\zeta z$, where $\zeta^m=1$, gives
$$
 f(z)=z^{m-2}h(z^m)
$$
for a holomorphic function $h$. Hence
$$
 q=\frac{1}{m^2}\frac{h(w)}{w}\,dw^2,
$$
which proves the asserted descent and pole order.

Since $M$ is quasiplatonic, the quotient by its full automorphism group is a sphere with three
orbifold points $p_1,p_2,p_3$. Consequently,
$$
 \widetilde q\in
 H^0\!\left(\mathbb P^1,
 K_{\mathbb P^1}^2(p_1+p_2+p_3)\right).
$$
The line bundle in this expression has degree $-4+3=-1$,
and therefore has no nonzero holomorphic sections. Thus $\partial F_M=0$. Since $F$ is
real-valued, its $(0,1)$-part vanishes as well, and hence $dF_M=0$.
\end{proof}

At the Bolza and Klein points, the quotient orbifolds in \eqref{eq:bolza-stratum} and
\eqref{eq:klein-stratum} admit Belyi maps to the three-cone full quotients of signatures
$(0;2,3,8)$ and $(0;2,3,7)$, respectively. The same pullback calculation as in
\cite[Lemma~4.1]{KalvinASNS}, together with the singular anomaly formula, makes the individual
quotient first variations explicit. We do not record them here:
Proposition~\ref{prop:bolza-klein-critical} and the identities \eqref{eq:bolza-stratum-variation},
\eqref{eq:klein-stratum-variation} imply that their weighted sums vanish at the symmetric points.

\begin{remark}\label{rem:normalization}
The $V_4$-relations \eqref{eq:bolza-stratum} and \eqref{eq:klein-stratum} hold throughout the
corresponding equisymmetric components and require no Belyi maps. At the symmetric points considered
here, compatible Belyi maps provide explicit metric uniformizations and determine the quotient
determinants and their first variations.

The determinant values found here provide the normalization required to integrate explicit
first-variation formulas from the symmetric points. This information cannot be recovered from a
formula for $\partial\bar\partial\log\det\Delta$, which determines $\log\det\Delta$ only up to a
pluriharmonic function.
\end{remark}

\appendix

\section{An independent \texorpdfstring{$V_4$}{V4}-check}
\label{app:klein-v4-check}

The $V_4$-relation \eqref{eq:klein-stratum} gives a second route to $\det\Delta_K$. It is
considerably less economical than the $PSL(2,7)$-calculation in Section~4. We include a sketch
because it provides an independent check of \eqref{eq:klein-final-expanded}. We use the Ciani model
and the $V_4$-action described in \cite{Lachaud}:
$$
 K:\quad x^4+y^4+z^4+
 a(x^2y^2+y^2z^2+z^2x^2)=0,
 \qquad a^2+3a+18=0.
$$
The two roots are complex conjugate and produce conjugate models and maps. The corresponding
hyperbolic metrics are isometric, so the determinant calculation is independent of this choice. In
this model,
$$
 V=\{1,\sigma_x,\sigma_y,\sigma_z\},
 \qquad
 H_1=\langle\sigma_x\rangle,\quad
 H_2=\langle\sigma_y\rangle,\quad
 H_3=\langle\sigma_z\rangle,
$$
where $\sigma_x,\sigma_y,\sigma_z$ change the sign of the corresponding homogeneous coordinate.
The coordinate permutations permute $H_1,H_2,H_3$; hence the three quotient tori $K/H_j$ are
isometric. Denote any one of them by $T$, and put $S_K=K/V$. The signatures and areas of the
quotients are
$$
\begin{array}{ccl}
 T
 & \text{hyperbolic torus with four cone points of angle }\pi,
 & \Area(T)=4\pi,\\
 S_K
 & \text{hyperbolic sphere with six cone points of angle }\pi,
 & \Area(S_K)=2\pi.
\end{array}
$$

The quotient sphere is the conic
$$
 S_K:\quad
 X^2+Y^2+Z^2+a(XY+YZ+ZX)=0,
 \qquad [X:Y:Z]=[x^2:y^2:z^2].
$$
The sign changes and coordinate permutations generate a subgroup $N\simeq S_4$ of
$\operatorname{Aut}(K)$, with $V\triangleleft N$. Hence the residual group $N/V\simeq S_3$
acts on $S_K$ by permuting $X,Y,Z$. Set
$$
 \omega=e^{2\pi i/3},
 \qquad
 \mu=\frac{13-3a}{252}.
$$
On the chart $X+Y+Z=1$, choose the rational coordinate $s$ by
\[
 \begin{split}
 X&=\frac13+s+\frac{\mu}{s},\\
 Y&=\frac13+\omega s+\frac{\mu\omega^2}{s},\\
 Z&=\frac13+\omega^2s+\frac{\mu\omega}{s}.
 \end{split}
\]
The $S_3$-invariant function $q=XYZ$ represents the quotient map
$$
 S_K\longrightarrow S_K/S_3\simeq\mathbb P^1_q
$$
of degree $6$. In the coordinate $s$, it is
$$
 q(s)=s^3+\frac{(13-3a)^3}{252^3s^3}
      +\frac{9a-11}{756}.
$$

The two quotient metrics are pullbacks of the same hyperbolic three-cone sphere
$\mathbb P^1(2,3,7)$. The required maps factor as
$$
 T\xrightarrow{\pi}S_K\simeq\mathbb P^1_s
 \xrightarrow[\deg 6]{q}\mathbb P^1_q
 \xrightarrow[\deg 7]{\mathcal B}\mathbb P^1_z.
$$
Here $\pi:T=K/H_j\to S_K=K/V$ is the quotient map induced by $H_j\subset V$; its deck group is
$V/H_j\simeq\mathbb Z_2$. Thus $f_S=\mathcal B\circ q$ is used for the sphere factor and
$f_T=f_S\circ\pi$ for the torus factor. To write the last arrow explicitly, set
$$
 h(q)=21(14-a)q-(a+18),
$$
$$
 p(q)=(270-a)+(44492+5782a)q+
       (313502+71687a)q^2.
$$
Then the degree-seven Belyi map is
$$
 \mathcal B(q)=
 \left(\frac{h(q)}{h(0)}\right)^7
 \left(\frac{p(0)}{p(q)}\right)^3.
$$
Consider the divisor
$$
 \boldsymbol\beta=
 -\frac67\cdot0-\frac12\cdot1-\frac23\cdot\infty
$$
and let $\widehat m_{\boldsymbol\beta}$ be the unit-area metric of Gaussian curvature $-\pi/21$
on $\mathbb P^1_z$ representing $\boldsymbol\beta$. The composite $f_S=\mathcal B\circ q$ has
degree $42$ and passport $[\,7^6;\,1^6 2^{18};\,3^{14}\,]$. In the coordinate $s$ on $S_K$,
the double cover $\pi:T\to S_K$ has the model
$$
 W^2=s^4-\frac13s^3+
 \left(\frac19-\mu\right)s^2-\frac{\mu}{3}s+\mu^2.
$$
This elliptic curve is also isomorphic to
$$
 T:\quad v^2=w(w-1)(w+a+1);
$$
see \cite[Example~3]{Beauville}. Since $a^2+3a+18=0$, the standard Legendre formula gives
$j(T)=-3375$. Thus $T$ has complex multiplication of discriminant $-7$, and its normalized
period ratio may be chosen as $\tau=\frac{1+i\sqrt7}{2}$. Accordingly, $f_T$ has degree $84$
and passport $[\,7^{12};\,1^4 2^{40};\,3^{28}\,]$. Under a holomorphic pullback, a cone order
$\beta$ at a target point becomes $d(\beta+1)-1$ at a preimage of local degree $d$. The two
passports therefore show that all preimages of $0$ and $\infty$, as well as all ramified
preimages of $1$, are smooth. Only the unramified preimages of $1$ remain cone points of order
$-1/2$: there are six on $S_K$ and four on $T$. Multiplication by $\pi/21$ changes the
curvature to $-1$ and gives areas $2\pi$ and $4\pi$, respectively. Hence
$$
 \frac{\pi}{21}f_S^*\widehat m_{\boldsymbol\beta},
 \qquad
 \frac{\pi}{21}f_T^*\widehat m_{\boldsymbol\beta}
$$
are precisely the quotient hyperbolic metrics described above.

Applied to $f_S$, the Belyi-pullback formula \cite[Theorem~2.1]{KalvinASNS} produces a closed
explicit formula for $\det\Delta_{S_K}$. For $T$, the singular anomaly formula
\eqref{eq:conical-anomaly} compares the hyperbolic pullback metric with the flat metric
$\lvert\omega_T\rvert^2$, where $\omega_T=ds/W$. Its reference determinant follows from
$\tau=(1+i\sqrt7)/2$ and the Kronecker limit formula. Hence $\det\Delta_T$ is also found in
closed explicit form. Finally, at the Klein point, \eqref{eq:klein-stratum} reduces to
\be\label{eq:klein-v4-check-relation}
 \det\Delta_K
 =
 \frac{(\det\Delta_T)^3}{(\det\Delta_{S_K})^2}.
\ee
Substitution of these independently evaluated quotient determinants into
\eqref{eq:klein-v4-check-relation} reproduces the closed explicit formula
\eqref{eq:klein-final-expanded}. No quotient determinant or normalization from the
$PSL(2,7)$-calculation enters this argument, and the agreement therefore gives an independent
check of that computation.

\end{document}